\documentclass[11pt,oneside]{amsart}

\makeatletter
\AtBeginDocument{\let\hl\@firstofone}
\makeatother

\usepackage[english]{babel}
\usepackage{fullpage}
\usepackage{amssymb,pstricks,amscd,epsfig}
\usepackage{graphicx}
\usepackage{xypic}
\usepackage{psfrag}
\usepackage{setspace}
\usepackage{paralist}
\usepackage{xparse}
\usepackage{mathbbol}

\usepackage[colorlinks,allcolors=red]{hyperref}

\usepackage{amsmath, amsthm, amssymb}
\usepackage{pdfsync}
\usepackage[latin1]{inputenc}
\usepackage{stmaryrd}
\usepackage{color}
\usepackage{setspace}
\usepackage[matrix,arrow,tips,curve]{xy}

\numberwithin{equation}{section}

\makeatletter
 \renewcommand\section{\@startsection {section}{1}{\z@}%
     {-4.5ex \@plus -1ex \@minus -.2ex}%
     {2.3ex \@plus.8ex}%
    {\centering\scshape}}

\def\A{\mathcal{A}}
\def\M{\mathcal{M}}

\def\C{\mathcal{C}}
\def\D{\mathcal{D}}
\def\E{\mathcal{E}}
\def\F{\mathcal{F}}

\def\H{\mathcal{H}}
\def\I{\match{I}}
\def\P{\mathcal{P}}
\def\Q{\mathcal{Q}}
\def\X{\mathcal{X}}
\def\Y{\mathcal{Y}}

\def\cO{\mathcal{O}}
\def\R{\mathcal{R}}

\def\V{\mathcal{V}}
\def\U{\mathcal{U}}
\def\K{\mathcal{K}}
\def\Y{\mathcal{Y}}
\def\W{\mathcal{W}}
\def\Z{\mathcal{Z}}
\def\P{\mathcal{P}}
\def\I{\mathcal{I}}

\def\AA{\mathbb{A}}
\def\PP{\mathbb{P}}
\def\CC{\mathbb{C}}
\def\ZZ{\mathbb{Z}}
\def\QQ{\mathbb{Q}}
\def\RR{\mathbb{R}}
\def\GG{\mathbb{G}}

\def\x{\mathbf{x}}

\def\ov#1{\overline{#1}}

\def\codim{\mathrm{codim}}

\def\Pic{\operatorname{Pic}}
\def\Hom{\operatorname{Hom}}
\def\Rhom{\operatorname{RHom}}
\def\Kom{\operatorname{Kom}}

\def\RG{\operatorname{R\Gamma}}

\def\Ext{\operatorname{Ext}}
\def\spec{\operatorname{Spec}}
\def\supp{\operatorname{Supp}}
\def\Stab{\operatorname{Stab}}
\def\Rep{\operatorname{Rep}}

\def\Lie{\operatorname{Lie}}
\def\Aut{\operatorname{Aut}}
\def\Rep{\operatorname{Rep}}

\def\ch{\operatorname{ch}}
\def\td{\operatorname{td}}

\def\GL{\operatorname{GL}}
\def\Amp{\operatorname{Amp}}

\def\rk{\operatorname{rk}}
\DeclareMathOperator{\gl}{\mathfrak {gl}}
\def\Stab{\operatorname {Stab}}
\def\Def{\operatorname {Def}}

\def\coh{\operatorname {Coh}}

\def\ext{{\E}xt}
\def\hom{{\H}om}
\def\tb{\text{\tiny$\bullet$}}

\def\ftb{\text{\footnotesize$\bullet$}}
\def\wt{\widetilde}
\def\wh{\widehat}
\def\oc2{\mathcal{O}_{\C_2}}
\def\blm{BLM${}^+$19}

\newcommand\rep{\operatorname{Rep}}
\DeclareMathOperator{\slope}{slope}
\def\Ku{\operatorname{Ku}}

\def\num{\operatorname{num}}
\def\blm{BLM${}^+$19}

\newtheorem{theorem}{Theorem}[section]
\newtheorem{prop}[theorem]{Proposition}
\newtheorem{lem}[theorem]{Lemma}

\newtheorem{rem}[theorem]{Remark}

\newtheorem{cor}[theorem]{Corollary}

\newtheorem{claim}[theorem]{Claim}

\NewDocumentCommand{\dslash}{s}{%
  \IfBooleanTF{#1}
    {\big/\mkern-7mu\big/}
    {/\mkern-6mu/}%
}
\newcommand{\be}{\begin{equation}}
\newcommand{\ee}{\end{equation}}

\begin{document}

\author{Enrico Arbarello and  Giulia Sacc\`a}

\title{Singularities of Bridgeland moduli spaces for K3 categories:  an update}

\begin{abstract} This survey is a continuation
of the study undertaken  in \cite{AS18}.  We examine the local structure of Bridgeland moduli spaces $M_\sigma(v,\D)$, where the relevant triangulated category $\D$ is either the bounded derived category $\D=\D^b(X)$ of  a K3 surface $X$, or the  Kuznetsov component $\D=\Ku(Y)\subset \D ^b(Y)$ of a smooth cubic fourfold $Y\subset \PP^5$. For these moduli spaces, building on \cite{Bmm19},
\cite{Bmm21} we give a direct proof of formality and, using their local isomorphism with quiver varieties,  we establish their
normality and their irreducibility, as long as $\sigma$ does not lie on a totally semistable wall. We then connect the variation of GIT quotients 
for quiver varieties with the changing of stability conditions on moduli spaces.

\end{abstract}

\dedicatory{In memory of Alberto Collino}

\maketitle
\tableofcontents

\section*{Introduction.}\label{intro} 
Let
$\D$ denote either the bounded derived category $\D=\D^b(X)$ of  a K3 surface $X$, or the Kuznetsov component $\D=\Ku(Y)\subset \D ^b(Y)$ of a smooth cubic fourfold $Y\subset \PP^5$. 
The aim of this survey is to give an account of various developments in the local analysis of moduli spaces of objects in $\D$. This analysis started  with the seminal  work of Kaledin, Lehn and Sorger \cite{KLS06},  \cite{KL07}. In  \cite{AS18} we followed the path they traced,  we introduced the so-called  ext-quiver, and we showed, among other things, that the variation of GIT quotients for quiver varieties governs the (partial) resolution of  singularities of the  moduli spaces $M_H(v, X)$ of pure dimension one sheaves  when they undergo a change of the polarization $H\in \Amp(X)$. There are two constituents  to that local analysis:  Luna's  \'etale slice theorem for the Quot scheme and the quadraticity (or, even strongly, the formality) of the Kuranishi family. The first  provides the \'etale local picture of moduli spaces of sheaves on a K3 surface, obtained as global quotient of the Quot scheme; the second  builds the bridge between this local picture and the ext-quiver. For  the case of sheaves on a K3 surface, the formality  theorem is the result of  several contributions (see Section \ref{kura-fa} for references) and was finally established by Bandiera, Manetti and Meazzini \cite{Bmm19}, and \cite{Bmm21} (see also \cite{BZ19}).
The picture changes substantially when passing to Bridgeland moduli spaces. Let $M_\sigma(v,\D)$
be the moduli space of $\sigma$-semistable objects with Mukai vector $v$, with respect to a Bridgeland stability condition $\sigma$ on $\D$.
These spaces  are not constructed as a GIT quotients of a Quot scheme. Indeed, their construction, in full generality, that is for arbitrary $\sigma$ and $v$, has only been accomplished rather recently through the groundbreaking work by Alper,  Hall, Rhyd, Halpern-Leistner,  Heinloth and Bayer, Lahoz, Macr\`i, Nuer, Perry, and  Stellari  \cite{BLM${}^+$19} . In Section \ref{good-mod} we ease into  an overview of these fundamental ideas. Regarding the formality of the Kuranishi family for Bridgeland moduli spaces,
we will give  direct and rather simple proofs based on the work by Bandiera, Meazzini, and  Manetti \cite{Bmm19} for the K3 case,
and on ideas of Kaledin \cite{Ka07} and Lunts \cite{Lunts08} for the Kuznetsov case. We should mention that a very general statement of formality for 2-Calabi Yau categories (and therefore including the two cases we just mentioned)  has been proved by Davison in \cite{Dav21}.  In Section \ref{norm} we address the question of normality for Bridgeland moduli spaces. Normality has been previously established, under the hypothesis  that the stability condition $\sigma$ be $v$-generic, in \cite[Theorem 3.10]  {BM-mmp14} for the K3 case and \cite[Corollary 2.4]{Sa23} for the case of the Kuznetsov component. When $\sigma$ lies on a $v$-wall,  the situation is more complicated. However, we succeed in proving that 
$M_\sigma(v,\D)$ is normal, even when $\sigma$ lies on a $v$-wall, as long as $\sigma$ does not lie on a totally semistable wall. In the proof of this theorem we take full advantage of the local  ext-quiver description of moduli spaces. A direct consequence of the normality result  is the irreducibility of $M_\sigma(v,\D)$, again as long as $\sigma$ does not lie on a totally semistable wall.
In Section \ref{VGIT}, we proceed  to develop the local analysis of Bridgeland moduli spaces by interpreting the passage of a stability condition from a given stratum of a wall to a shallower one, in terms of variation of GIT quotients for  the ext-quiver variety (see also \cite{Hal21}). In this regard, the context provided 
by the above mentioned paper \cite{AHLH22} makes the treatment much simpler than the one 
we described in the case of sheaves in \cite{AS18}.
\vskip 0.1 cm
\begin{rem}\label{G-M}{\rm It should be noticed that all the results in this paper concerning the Kuznetsov component of a cubic fourfold, work verbatim for the Kuznetsov component
of a Gushel-Mukai variety in dimension 4 or 6, provided one replaces the results 
in \cite{BLM${}^+$19}, with the ones in \cite{PPZ19}.}
\end{rem}
\vskip 0.2 cm
{\bf Acknowledgments}. We warmly  thank Jarod Alper, Jochen Heinloth, Emanuele Macr\`i, and  Francesco Meazzini for  valuable conversations on the subject of this note, and the referee for useful comments. We also thank Laura Pertusi for sharing the manuscript \cite{CPZ21}.The second named author is partially supported by NSF CAREER
grant DMS-214483, and NSF FRG grant DMS-2052750.

\section{Notation}\label{notation}
We work over the field of complex numbers: $k=\CC$.
\vskip 0.2 cm

We let $\D$ be as in the Introduction.
For  Kuznetsov components 
we  adopt the notations of  Part IV, Section 29, p. 295 in \cite{BLM${}^+$19}, and  also in  \cite{BLMS17} 
We will study {\it numerical}  stability conditions on $\D$, w.r.t. a lattice $\Lambda$.
In the $K3$-case, we set $\Lambda=H^*_{alg}(X,\ZZ)=K_{num}(X)$, while for Kuznetsov components we set 
$\Lambda =K_{num}(\Ku(Y))$, where $K_{num}$ denotes the numerical Grothendieck group  \cite[Def 2.20]{MSt19}. The Mukai-vector homomorphism $v: K_0(\D)\to K_{num}(\D)$ is given by $v(E)=\ch(E)\cdot\sqrt{\td(X)}$ in the K3 case, and by $v(E)=\ch(E)\cdot\sqrt{\td(Y)}$ in the Kuznetsov case.

\vskip 0.2 cm
We recall that  a {\it numerical} stability condition $\sigma$ on $\D$, consists in a pair 
$(\A, Z)$ where $\A\subset \D$ is the heart of a $t$-structure, and $Z$ is a stability function, i.e.
an additive  function $Z: K_0(\A)\longrightarrow \CC$, factoring through $\Lambda$, in the sense that
for $E\in \A$, then $Z(E)=Z(v(E))$. 
A stability function should satisfy the following property:
 $$
 0\neq E\in \A\qquad\Rightarrow\qquad Z(E)=Z(v(E))\in \RR_{>0}e^{i\pi\phi(E)}\,,\qquad 0<\phi(E)\leq 1
 $$
 The number $\phi(E)$ is the {\it phase} of $E$. The
{\it slope} $\nu(E):=\Re Z(E))/\I m Z(E)$ is related to the phase
by the equation: $\phi(E)=\frac{1}{\pi}\mbox{arccot}(-\nu(E))$.
 An object $E\in \A$ is said to be $\sigma$-{\it semistable} if there is no non-zero sub-object of $E$ with larger phase. A (numerical) stability
 condition should also satisfy the Harder-Narasimhan  property: every $E\in \A$ admits a filtration 
 $$
 0=E_0\subset E_1\subset E_2\subset\cdots\subset E_{n-1}
 \subset  E_n=E
 $$
 such that $E_j/E_{j-1}$ is $\sigma$- semistable, and   the sequence of phases $\{\phi (E_j/E_{j-1})\}$
 is strictly decreasing. Finally  a stability condition should  satisfies the {\it support property}: $\exists$ a norm $||\cdot||$  on $\Lambda$, such that
$|Z(v(E))|\geq c||v(E)||$, 
  for every$\sigma$-semistable $E$
 on $\Lambda$ e.g. \cite[Section 4]{MS17}.

 Bridgeland endows the set $\Stab(X)$ of stability conditions on $X$ with a topology making
 $\Stab(X)$ into a complex manifold. 
 Bridgeland's theorem states that
the natural projection
\be\label{Z-stab}
\aligned
\Z: &\Stab(X)\longrightarrow \Hom(\Lambda_\RR,\CC)\\
&\sigma=(\A, Z)\quad\mapsto \quad Z
\endaligned
\ee
is a local homeomorphism  onto its image.  
By definition, given a numerical stability condition $\sigma=(\A, Z)$, there exists
$\Omega_\sigma\in  \Lambda_\RR$ such that
\be\label{z-omega}
Z(E)=\langle\,\Omega_\sigma, v(E)\,\rangle:=-\chi(\Omega_\sigma, v(E))
\ee
It is also convenient to introduce the class $\omega_\sigma=-\Omega_\sigma^\vee$
so that
\be\label{z-omega2}
Z(E)=\chi(\omega_\sigma\otimes E)
\ee

It is  useful to pass from a  given stability condition to an "equivalent" one
where computations might be  easier. To do so, one sees a stability condition $\sigma$, as a pair  $\sigma=(\P, Z)$, where $Z:K_0(\D)\to \CC$ is a homomorphism,  and where $\P$ is a {\it slicing} of $\D$, i.e.  a collection of full (a posteriori abelian) subcategories $\{\P(\phi)\subset \D\}$, one 
for each $\phi\in \RR$, satisfying a certain number of properties \cite{Br08}, Definition 2.1 p.246]. The objects of $\P(\phi)$ are said to be $\sigma$-semistable of phase $\phi$, and for $E\in\P(\phi)$ one has $Z(E)=Z(v(E))\in \RR_{>0}e^{i\pi\phi(E)}$. The heart $\A$ can be recovered as the subcategory $\P(0,1]$ of objects whose  Harder-Narashiman factors have phases 
lying in the interval $(0,1]$. 
The universal cover $\wt\GL^+(2, \RR)$ of $\GL^+(2, \RR)$ acts on $\Stab(X)$
in the following way. Think of an element $\wt g\in  \wt\GL^+(2, \RR)$ as  a pair $\wt g=(f, g)$, where $g\in\GL^+(2, \RR)$, and $f:\RR\to\RR$ is an increasing function s.t. $f(\phi+1)=f(\phi)+1$, satisfying  the compatibility condition that: $g_{|S^1}:S^1\to\RR^2\smallsetminus\{0\}/\RR^+$ coincides with $f:S^1=\RR/\ZZ\to\RR/\ZZ$.
Then $\wt g\cdot (Z,\P)= (Z',\P')$, where $Z'=g^{-1}\circ Z$ and $\P'(\phi)
=\P(f(\phi))$. Given a stability  condition $\sigma=(Z, \A)$, and a Mukai vector $v$, by acting with $\wt\GL^+(2, \RR)$, one can always assume that $Z(v)=-1$, if one wishes. \vskip 0.2 cm
Finally we will always restrict Bridgeland's local homeomorphism (\ref{Z-stab}) to the 
$\wt\GL^+(2, \RR)$-invariant connected component
$$
\Stab^\dagger(\D)\subset \Stab(\D)
$$
containing the explicitly constructed stability conditions $\sigma_{\beta, \alpha}$, of \cite{Br08}, and  \cite{BLMS17}.

\vskip 0.3 cm

\section{Moduli stacks and moduli spaces}\label{good-mod}

\subsection{The classical picture}\label{classic}

Let us briefly collect the ingredients that go into the construction of the {\it classical moduli spaces of sheaves on a K3 surface} $X$.
\vskip 0.5 cm

{\bf 1)} First we need a  Mukai vector $v\in H^*_{alg}(X,\ZZ)$ 
and  polarization $H\in \Amp(X)$.
These two data determine, and are determined by, the Hilbert polynomial of $F$:
$$
P_{H}(F)(t):=\chi(F\otimes H^t)=\alpha_d(F)t^d+\cdots+\alpha_0(F)\,, \qquad2\geq d=\dim\supp (F)
$$
which in turn defines the notion of $H$-{\it Gieseker semistability}: $F$ is semistable
if for  $E\subset F$, then  $p_{H}(E)\leq p_{H}(F)$, where $p_{H}(-):=P_{H}(-)/\alpha_d(-)$
is the {\it reduced Hilbert polynomial} and where the order is lexicographical.
 
 For a fixed rational polynomial $P\in \QQ[t]$,  the Quot scheme $Q _{H,P}$  parametrizes equivalence classes of surjections
$q_F: V\otimes\cO_X(-m)
\twoheadrightarrow F$, where $P_{H}(F)=P$, where $V$ is a vector space isomorphic to $H^0(F(m))$, for some  large $m$, and where $q_F\sim q_{F'}$,  if $q_F=\alpha q_{F'}$, for some isomorphism $\alpha: F\cong F'$.
\vskip 0.5 cm

{\bf 2)}  Let $Q_0^{(H\text{-}Gies)_{ss}}$ be the  open subset of the quot-scheme $Q _{H,P}$ whose points
$[q_F: V\otimes\cO_X(-m)
\twoheadrightarrow F]$ are such that 

a) $F$ is $H$-Gieseker-semistable, and 

b) $V\cong H^0(F(m))$. 

The general linear group $\GL(V)$ acts naturally on $Q_0^{(H\text{-}Gies)_{ss}}$.
The Moduli space of  $H$--Gieseker semistable sheaves with Mukai vector  equal to $v$
is defined as 
\be\label{def-mod-quot}
M_H(v):=Q_0^{(H\text{-}Gies)_{ss}}/GL(V)
\ee
To make any sense out of this quotient one needs the GIT arsenal, and the first ingredient
is a $\GL(V)$-linearized line bundle $L$ on $Q _{H,P}$. This line bundle  is provided by the determinal line bundle 
\be\label{lb-quot}
L_H=\det(p_*(\F\otimes q^*\cO_S(\ell H))
\ee
where $\F$ is a universal family over $Q _{H,P}\times X$, while $p$ and $q$ are the 
natural projections, and $\ell$ is a fixed large integer. 
It now make sense to talk about 
$\GL(V)$-semistable points with respect to $L_H$, or briefly $L_H$-semistable points. 
It is only for these points that a GIT quotient makes sense.

\vskip 0.5 cm
{\bf 3)}  Thus the fundamental step in giving sense to the quotient (\ref {def-mod-quot})   is to show that points in ${Q_0^{(H\text{-}Gies)_{ss}}}$ are also
 $L_H$-GIT-semistable. But more is true:

\be\label{two-stab}
Q_0^{(H\text{-}Gies)_{ss}}=\left(\ov {Q_0^{(H\text{-}Gies)_{ss}}}\right)^{{(L_H\text{-}GIT})_{ss}}\subset Q^{{(L_H\text{-}GIT})_{ss}}
\ee

  To decide wether a point
\be\label{point-in-Q}
x=[q: V\otimes\cO_X(-m)
\twoheadrightarrow F]\in \ov {Q_0^{(H\text{-}Gies)_{ss}}}
\ee
is  $L_H$-GIT-semistable, one uses the Hilbert-Mumford criterion, and looks at one-parameter subgroups $\lambda:\GG_m\to\GL(V)$. For each such, one looks at
the limit $\lim_{t\to 0}\lambda(t)x=x_0$, exhibiting $\GG_m$ as a subgroup of $G_{x_0}:=\Stab(x_0)$. Thus, $\GG_m$ acts, via $\lambda$, on the fiber $(L_{H})_{x_0}$. Via the $\GL(V)$-linearization of $L_H$, one can write the $\GG_m$-action on that fiber  as 
$\lambda(t)v=t^rv$. Setting $r=\mu^{L_H}(x,\lambda)$,  the criterion says $x$ is $L_H$-semistable if and only if $\mu^{L_H}(x,\lambda)\geq 0$, for every $\lambda$.
To prove (\ref{two-stab}), one needs a bridge between $H$-Gieseker-semistability (computed in terms of Hilbert polynomials) and $L$-semistability. This  bridge is provided by expressing the invariant $\mu^L(x,\lambda)$ in terms of Hilbert-polynomials. It goes as follows. Fix a point $\x \in  \ov {Q_0^{(H\text{-}Gies)_{ss}}}$ as in (\ref{point-in-Q}).
A one-parameter group $\lambda$ determines
a decomposition of $V$ into {\it weight  subspaces}, and consequently a decreasing  filtration $\{V^w\}$  for $V$, and, via $q$, a decreasing filtration for   $F$: 
\be\label{filtr}
F^\bullet:=q(V^\bullet\otimes \cO_X(-m))):\qquad\qquad \cdots\subseteq F^{w+1}\subseteq F^{w}\subseteq\cdots\subseteq F
\ee
The associated graded pieces of $\{F^w\}$ are denoted by $F_w$.  The formula, bridging Gieseker stability with GIT stability, and  expressing 
$\mu^{L_H}(x,\lambda)$ in terms of Hilbert polynomials is the following  (recall the definition (\ref{lb-quot})):
\be\label{mu}
\mu^{L_H}(x,\lambda)=\sum_{w\in \ZZ} wP_{H, F_w}(\ell)\,.
\ee
From this, the equality (\ref{two-stab}) follows, after some highly non-trivial work.

  \begin{rem}{\rm A notational remark: the datum of a one-parameter subgroup through the point 
$x$ (as in (\ref{point-in-Q})),
determines a decreasing filtration  (\ref{filtr}).
Discarding set-equalities from that filtration we are left with a {\it finite, strictly decreasing} filtration:
$0\neq F^{w_p}\subset\cdots\subset F^{w_0}=F$, with $w_p>\cdots>w_0$ in $\ZZ$.
Following \cite{HL10}, we change  notation, and rewrite this datum as a {\it weighted filtration}
$(F^\bullet, w_\bullet)$, where now 
\be\label{filt2}
F^\bullet= \{ 0=F^{p+1} \subsetneq F^{p}\subsetneq\cdots\subsetneq F^{0}=F\}\,,
\qquad w_\bullet=\{w_p>\cdots>w_0\}\,,\qquad F_i=F^i/F^{i+1}
\ee
It should be noted that filtrations as in (\ref{filtr}) are not {\it arbitrary} filtrations of $F$.
For example, if condition {\bf 2)}, b) holds, then all the terms of a filtration of type 
(\ref{filtr}) are sheaves generated by their sections.}
\end{rem}
 Built  in the above construction of $M_H(v)$, are  three results. The fact that the scheme
$M_H(v)$ is  {\it proper}, that it is {\it separated}, and that its points 
 represent {\it $S$-equivalence classes} of sheaves.

\vskip 0.5 cm\
\subsection{Road map to moduli}\label{road}

We will now give a glimpse of the  path toward  the definition of  Bridgeland moduli spaces. 
This path was laid by the work of many authors (e.g.: 
Abramovich, Alper, Bayer, Hall, Halpern-Leistner, Heinloth, 
 Lahoz, Lieblich, Macr\`i, Perry,  Neur, Polishchuk,  Rydh, 
 and Stellari,
\cite{Lie06}, 
\cite{AP06},
\cite{P07},
\cite{Alp13}, 
\cite{Hein17},
 \cite{AHR20},
 \cite{Hal21},
 \cite{Hal22}, 
\cite{AHLH18},
\cite{BLM${}^+$19})
Starting from a moduli problem, the road envisioned by the above authors consist essentially in the four
steps below. We will often refer to Alper's Lecture notes \cite{Alp21} for basic definitions on stacks (e.g. for the definition of  $|\X|$, see Definition 2.2.6.  in loc.cit.), and we will use interchangeably the terms {\it algebraic stack} and {\it Artin stack}. Here are the four steps.

\vskip 0.3 cm

{\bf A)} Find an algebraic stack $\X$ equipped with and an appropriate theory substituting the classical GIT and which is related to the moduli problem at hand.
This will be the replacement of the Quot-scheme, while the "appropriate theory" will be a replacement of the $L_H$-GIT theory in item {\bf 2)} above.
A partial  substitute for the "big group", 
 $\GL(V)$ as a host of all the stabilizers of points in the Quot scheme will be the inertia stack $\I$ defined by the cartesian diagram
$$
\xymatrix{
\I\ar[r]\ar[d]_p&\X\ar[d]\\
\X\ar[r]^\Delta&\X\times\X\\
}
$$
It hosts all the stabilizers of points
$x\in |\X|$ as fiber of $p$.  These stabilizers are denoted by 
$\Aut_\X(x)$, or by $\Stab(x)$, or by $G_x$.

{\bf B)}  Show that there   is an open, quasi-compact algebraic subtack  $\X^{ss}\subset\X$
parametrizing semistable points for the above substitute of GIT. Again $''ss''$ replaces  $''ss_{L_H}''$,
in the classical picture.

{\bf C)}  Show that there is a {\it good moduli space} $\X^{ss}\to Y$ for the algebraic stack $\X^{ss}$, and that $Y$ is a complete, separated scheme, and furthermore that $Y$ parametrizes 
$S$-equivalence classes in the above "new GIT" sense (see Subsection \ref{exist-good})

{\bf D)}  Show that $Y$ {\it also} parametrizes $S$-equivalence classes in the {\it original moduli problem},
not only in the "new GIT" sense.
This is the step equivalent to the proof of (\ref{two-stab}) in the classical case.

We shall start by explaining the crucial term {\it good moduli space} appearing in item {\bf C)}.

\subsection{Good moduli spaces, and the slice \'etale}\label{good-mod-slice}

Alper \cite{Alp13}, following the footsteps of \cite{AOV08}, introduces  an amazingly  simple notion of  {\it  { good moduli space}}: 
\be\label{def-gms}
\text{\parbox{.88\textwidth}{{\it A quasi compact morphism $\phi: \Y\to Y$ from a (locally noetherian) Artin stack to an algebraic space is a good moduli space if:\it \,\,{\bf(1)} The pushforward functor on quasi-coherent sheaves is exact (in the locally noetherian case, coherent sheaves are sufficient).\,\,\it {\bf (2)} The induced morphism  $\cO_Y\to\phi_*\cO_\Y$ is  an isomorphism }}}
\ee
\vskip 0.3 cm
The simplest example of a good moduli space is provided by a linearly reductive group $G$ acting linearly on a projective scheme $X$. Then the natural morphism
from the stack quotient to the GIT quotient : $[X/G] \to X\dslash G$ is a good moduli space for $[X/G]$. In the next section we will see the central role played by the stack $[\AA^1/\GG_m]$, in the entire moduli theory (in our case $\GG_m=\CC^*$).

A good moduli space $\phi: \Y\to Y$ enjoys all the relevant properties of a GIT quotient.
For example (but look at the  complete list in  \cite[p. 2351]{Alp13} ):

i) $\phi$ is surjective and universally closed, so that $Y$ has the quotient topology
(via the induced map $|\Y|\to Y$). 

ii)  Given two geometric  points $x_1,x_2\in \Y$, then $\phi(x_1)=\phi(x_2)$ if and only if 
$\ov{\{x_1\}}\cap\ov{\{x_2\}}\neq \emptyset$, (clearly pointing towards $S$-equivalence).

iii) $\phi$ is universal for map to algebraic spaces.

Condition {\bf (1)} in the definition of good moduli space, namely the exactness of the functor
$\phi_*: Q\coh(\Y)\to Q\coh(Y)$ is also called {\it cohomological affiness}. In case $\Y$ has affine diagonal it amounts to the vanishing of $R^i\phi_*\F$, for quasi coherent sheaves $\F$,
and for $i>0$  \cite[ Remark 3.5, p.2357]{Alp13}. The proof of part (ii) of Lemma 4.9 \cite[p.2365]{Alp13} is enlightening:  it shows how  cohomological affiness
implies the crucial property ii) above.

Under  mild restrictive hypotheses, the local structure of an Artin stack 
is governed by the stack version of Luna's slice \'etale theorem established by Alper, Hall and Rydh in \cite{AHR20}. Here are their beautiful results 
in  \cite[Theorems 1.2,  and  2.9]{AHR20}

\be \label{sl-et}
\text{\parbox{.88\textwidth}{{\it Suppose that $\Y$ is a quasi-separated algebraic stack, locally of finite type
over an algebraically closed field $k$.  Let $x=[\spec k\to \Y]\in \Y(k)\subset|\Y|$ be a geometric point, and assume that the stabilizer $G_x$ is linearly reductive},(we will always work under these assumptions, and moreover we assume $k=\CC$).   {\it Then:}\\
{\bf(a)} {\it There exist an affine scheme $U=\spec A$, acted on by  $G_x$, a $k$-point $w\in U$, fixed by $G_x$, and an \'etale morphism
$f:( [U/G_x],w)\longrightarrow (\Y,x)$.}\\
{\bf(b)}  {\it Furthermore, if  $\Y\to Y$ is a good moduli space then there is a cartesian diagram}
$$
\xymatrix
{[U/G_x]\ar[r]^\rho\ar[d]&\Y\ar[d]^\pi\\
U\dslash G=\spec(A^{G_x})\ar[r]^{\quad\qquad\epsilon}&Y
}
$$
{\it  where the horizontal maps  are  \'etale,  so that $\epsilon$  is an \'etale neighbourhood of 
$\pi(x)$.}}}
\ee

We come now to a very important property of good moduli spaces.

 {\bf $S$-equivalence:}  
 \be \label{s-eq}
\text{\parbox{.88\textwidth}{{\it Let  $\Y$ be algebraic stack $\Y$ over $k=\CC$, with affine stabilizers, 
and of finite presentation,  (properties that always hold in the cases we are interested in), and  
suppose that $\Y\to Y$ is a good moduli space. Then $Y$  parametrizes $S$-equivalence classes
in $|\Y|$.} }}
\ee
As usual $S$ stands for Seshadri. In this abstract setting, this means, by definition,  that $Y$ is the quotient of $|\Y|$ by the smallest equivalence relation identifying  $f(1)$ with $f(0)$, for every filtration $f:\Theta\to \X$. The proof of this crucial property
 \cite[Corollary 5.5.3]{Hal22}), is an easy consequence of the above mentioned 
 \'etale slice  theorem.

In the back of one's own mind one should think that $\Y=\X^{ss}$, as in point {\bf C)}
of the preceding subsection.

\vskip 0.2 cm
\subsection{The "one-parameter group" probe}\label{probe}
\vskip 0.2 cm
In his comprehensive work \cite{Hal22}, Halpern-Leistner sets the foundations for a theory, in the realm of algebraic stacks, which substitutes the classical Geometric Invariant Theory for schemes.   While in Mumford's theory, one-parameter subgroups appear as tools to probe an a-priori established concept of stability, in Halpern-Leistner's view, one-parameter subgroups lay at the foundation of the new theory. Given an aglebraic stack $\X$, the correct notion of a one-parameter group through a point $x\in |\X|$ is a map
\be\label{1-par}
f: [\AA^1/\GG_m]\longrightarrow \X\,,\qquad f(1)=x
\ee

Halpern-Leistner uses the notation $\Theta= [\AA^1/\GG_m]$, 
and this choice 
brings with it the terminolgy: {\it$\Theta$-reductivity, $\Theta$-stratification}, which we will briefly explain below. 
The topological space $|\Theta|$ has two points:
the closed point $0$ and the open point $1$, which is  the image of $1\in\AA^1$, under the projection
$\AA^1\to \Theta$.

A map like (\ref{1-par}), carries, and is determined by, a number
of informations:

a)  the object  $x=f(1)\in\X(k)$,

b)  the closed point  $x_0=f(0)\in|\X|$,

c) the induced homomorphism $\gamma: \GG_m\to G_{x_0}=\Aut_\X(x_0)$.

Going back to point {\bf 3)} of the  classical construction, a map $f$ as in (\ref{1-par})
is also called a {\it filtration of the point $x=f(1)$}, while  $f(0)$  is  the associated graded object. It is crucial  to simultaneously think of (\ref{1-par}), both
as one-parameter group and as a filtration. Let us explain why.

Let $X$ be a scheme.
A quasi-coherent sheaf $E$ on $\Theta\times X$  can be described as 
a $\GG_m$-equivariant sheaf on $\mathbb A^1\times X$
under the $\GG_m$-action on $\mathbb A^1$. 
Look  at the case
 $X=\spec B$.  Such an equivariant sheaf corresponds to a graded module $M$ over the graded ring $B[t^{-1}]$  \cite[Sec. 1]{Ga19},
and here we follow Halpern-Leistner convention about the degree of $t$:
\be\label{grad-mod}
M=\underset {i\in \ZZ}\oplus M^it^{-i}\,,\qquad M^i=0\,,\quad i>>0\,,\quad M^j=M^{j-1}\,,\quad j<<0\,,
\ee
 where $M^i$ is a $B$-module .
Via the Rees construction \cite[Section 5.1]{SaSh18} , \cite[Section 4]{Hal13} the datum of 
the graded module $M$ is equivalent to the datum of
a diagram 
$$
M^\bullet:\qquad\cdots\longrightarrow M^i\overset{\alpha_i}\longrightarrow M^{i-1}\longrightarrow\cdots 
$$ 
of $B$-modules, 
where the map $M^i\to M^{i-1}$ corresponds to multiplication by $t^{-1}$: 
$$
\alpha_i(v)t^{i-1}:=t^{-1}\cdot vt^i
$$
This of course corresponds to a  diagram of coherent sheaves on $X$: 
$$
E^\bullet:\qquad\cdots\to E^i\to E^{i-1}\to\cdots 
$$
This  description can be upgraded from $\coh(\Theta\times X)$ to $\D(\Theta\times X)$, and for any scheme $X$. 
As we shall be mostly interested in perfect complexes, we may assume that multiplication by $t^{-1}$ is injective, so that in the above diagrams $M^\bullet$, and $E^\bullet$  are in fact  filtrations.
Thus, we will view  perfect complexes in $\D(\Theta\times X)$ as
filtered  objects, and we will often write:
\be\label{e-bull}
E^\bullet:\qquad0\neq E^{w_p} \subsetneq E^{w_{p-1}}\subsetneq\cdots\subsetneq E^{w_0}=E\,,\quad w_p>\cdots >w_0
\ee
where the inclusion $E^{w_i} \subsetneq E^{w_{i-1}}$ is given by multiplication by $t^{w_{i-1}-w_i}$.  
Consider  the case where $X=\{pt\}$, and denote by $\cO_\Theta(w)\in \coh(\Theta)$ the rank one, locally free sheaf corresponding to the $\CC[t]$-graded module $\CC[t]\cdot t^w$. Setting
 
\be\label{def-u}
u=[\cO_\Theta(1)]
\ee
 we have $K_0(\Theta)=\ZZ[u^\pm]$. Now look back at the graded $B[t^{-1}]$-module $M$ in 
(\ref{grad-mod}). A simple computation in the $\GG_m$-equivariant  Grothedieck  group gives

$$
[M]= \left[\underset {i\in \ZZ}\oplus M^it^{-i}\right]=\left[\sum_i \CC[t^{-1}]\cdot t^{-i}\otimes[M^i-M^{i+1}])\right]=
\left[\sum_i u^{-i}\otimes\operatorname{gr}_i(M^\bullet)\right]
\in K_0(B[t])^{\CC^*}
$$
and in general
\be\label{usef-form-k}
[E]= \left[\underset {i\in \ZZ}\oplus E^it^{-i}\right]=
\left[\sum_i u^{-i}\otimes\operatorname{gr}_i(E^\bullet)\right]\in K_0(\Theta\times X)
\ee

\vskip 0.2 cm
Let us now consider a moduli stack $\M$ parametrizing perfect complexes  of a given type in $\D(X)$, and consider a morphism
\be\label{ftheta}
f: \Theta\longrightarrow \M
\ee

Given a universal object
$\E\,\,\text{over} \,\,\M\times X$,
we may look at 
$$
\E_f:=(f\times 1)^*\E\in \D(\Theta\times X)
$$
Suppose that, under  the above identification, we have $\E_f=E^\bullet$,  then
$E=f(1)$ and ${gr}E^\bullet=f(0)$
where $E$ and $E^\bullet$ are as in (\ref{e-bull})
 \cite [Lemma 6.3.1, with $n=1$]{Hal22}, \cite[Corollary 7.12] {AHLH18}.  From (\ref{usef-form-k})
we get the following useful formula  in $K$-theory:
\be\label{class-e}
[\E_f]=
\sum_ju^{-w_j}[\operatorname{gr_j}E^\bullet]\in K_0(\Theta\times \M)
\ee

\subsection{The notion of $\Theta$-stability for a stack}\label{GIT-stack}
In the classical construction, the {\it weight} of a point, with respect to a one-parameter subgroup,
is   defined in terms of a linearization of a line bundle over the Quot-scheme
(see point {\bf 3)} above).  Halpern-Leistner introduces an abstract notion of weight
as a numerical invariant $\mu$ assigning, to each filtration $f$, a value $\mu(f)$
in a totally ordered set $\Gamma$, equipped with a marked point $0\in \Gamma$
(e.g. $\ZZ$) \cite[Section 4.1, p.91]{Hal22}.

This function should satisfy a number of quite natural requirements \cite[Definition 0.0.3 and Definition 4.16]{Hal22}. A point $x\in |\X|$ will be called {\it $\Theta$-unstable} or simply {\it unstable} if there exists 
a filtration $f$, with $f(1)=x$ and $\mu(f)>0$. Points that are not unstable are called {\it $\Theta$-semistable} or simply {\it semistable}, and they are the support of a substack $\X^{ss_\mu}\subset \X$. One of the main tasks is 
to decide whether the substack $\X^{ss_\mu}$ is algebraic and whether it admits a good moduli space. 

The inspiring ideas for the definition of the numerical function $\mu$ originate in the groundbreaking work  by Kempf, 
\cite{Ke78}, (and especially \cite{Ke18}!), and also in the  works by Hessenlink, Ness and Mumford, \cite{Hes79}, \cite{NM84}. As in Kempf's theory, this numerical function brings with it the notion  of {\it Harder-Narasimhan filtration of a point $x\in |\X|$}, as a filtration
$f$ with 
$$
\mu(f)=M^\mu(x):=\sup\{\mu(f')\,\,|\,\,f'\,\,\text {a filtration s.t.}\,\,f'(1)=x\in|\X|\}.
$$
In Kempf's theory this filtration corresponds to  the "fastest" one-parameter group through $x$. It is of course crucial to insure uniqueness of a  Harder-Narasimhan filtration. Halpern-Leistner
 identifies the correct property that $\X$ should satisfy to ensure this uniquiness: he calls it {\it $\Theta$-reductivity}. In simple terms it can be stated as follows: given a discrete valuation domain $R$, and  a morphism $T=\spec R\to \X$, then any filtration on the generic point of 
$T$ extends to a family of filtrations on all of $T$ \cite  [Definition 5.11, and Theorem 5.1.9]{Hal22}.
To better explain this concept one should organize all possible filtrations for a stack $\X$ as  points of a new stack
$\operatorname {Filt}(\X)$ \cite [Definition 1.1.11 with $n=1$] {Hal22} and look at the evaluation morphism $\operatorname{ev}_1: \operatorname {Filt}(\X)\to \X$, defined by $\operatorname{ev}_1(f)=f(1)$. Then the condition of 
 $\Theta$-reductivity 
 for $\X$ is nothing but   the evaluative criterion for properness of the morphism $\operatorname{ev}_1$
 \cite[Definition 5.1.1]{Hal22} .

Via the totally ordered set $\Gamma$ one defines a {\it set-theoretical} stratification:
$$
|\X|_{\leq c}=\{x\in |\X|\,\,\,\,|\,\,\,\,M^\mu(x)\leq c\in \Gamma\}\subset |\X|
$$

When this stratification is the support of a bona fide stratification by open substacks $\X_{\leq c}\subset \X$, then we are in the presence of what Halpern-Leistner calls a {\it {\rm{(}}weak{\rm{)}} $\Theta$-stratification of $\X$}, 
\cite [Definition 2.2.1, Theorem 2.2.2, Definition 4.1.1]{Hal22}.  
This happens under two conditions: $\Theta$-reductivity and {\it "boundedness"}
(see  \cite[Theorem A, item (B)]{Hal22}. 
Of course the semistable stratum in this stratification  is given by $\X^{ss_\mu}=\X_{\leq 0}$.

\subsection{ Geometrical numerical functions}\label{num-funct}

The main examples of  numerical functions, for which one should test the above mentioned properties,
are given as follows. Recal  that $H^*(\Theta, \QQ)=H^*_{\GG_m}(\AA^1)\cong\QQ[u]$ (see (\ref{def-u})).
Now start from two  classes 
\be\label{classes}
\ell\in H^2(\X;\QQ)\,,\text{ and}\,\, b\in  H^4(\X;\QQ)
\ee

 Assume that $b$ is {\it positive} in the sense that for every filtration $f$, as in (\ref{1-par}), $f^*b=au^2$ with $a>0$. Then set
$$
\mu(f)=\frac{f^*\ell}{\sqrt{f^*b}}\in \QQ
$$
Because of the positivity of the denominator,   the numerator alone decides semistability (see e.g. \cite[Definition 6.13]{AHLH18}. Namely  a point $x\in \X$ is $\Theta$-unstable if there is a filtration $f:\Theta\to \X$, with $f(1)=x$, and $f^*(\ell)>0$. In fact we will often write $\X^{ss_\ell}$ instead of $\X^{ss_\mu}$.
On the other hand, the above quotient is invariant under  base change of the type
 $\Theta\overset{(\bullet)^n}\longrightarrow\Theta$, (this is  relevant, with respect to the \'etale topology of $\X$,  in view of  the Luna-type theorem discussed in Subsection \ref{good-mod-slice}).

\vskip 0.2 cm 

To give these very general notions a firmer ground, at least in one's own intuition, we have to wait for Subsection \ref{mod-stack}, where we will decrease the level of abstraction to describe the stacks that are immediately relevant to the construction of
Bridgeland moduli spaces.

The next task is to see under which conditions the stack $\X^{ss_\ell}$ admits a good moduli space.
 
\subsection {Existence of good moduli spaces. }\label{exist-good}
Alper, Halpern-Leistner, and Heinloth \cite {AHLH18} find necessary and sufficient conditions for an algebraic stack
to admit a good moduli space. The algebraic stacks $\Y$  we are interested in, are defined over $k=\CC$, 
are of finite type, and with affine diagonal. For these algebraic stacks their criterion is quite simple \cite [Remark 5.5, p. 29] {AHLH18}: 

\be\label{gms}
\text{\parbox{.85\textwidth}{\it $\Y$ admits a good moduli space $\pi: \Y\to Y$,   if and only if $\Y$ is $\Theta$-reductive, and $S$-complete.
}}
\ee
While $\Theta$-reductivity is a condition that filtrations extend under specialization, the condition of {\it $S$-completeness}
(here again, $S$ stands for Seshadri) formalizes the idea of semistable reduction, and again has the flavor of an evaluative criterion:
it says that, given a discrete valuation domain $R$, and  a morphism $T=\spec R\to \X$, then any  two families (of filtrations)  on $T$ that coincide over  the generic point differ by an {\it elementary modification}
$T$, \cite [Section 2B]{Hein17}, \cite[Definition 3.35]{AHLH18},  \cite[Definition 5.5.4]{Hal22} .
While $\Theta$-reductivity is tested by 
families (of filtrations)  of the form $\Theta\times\spec R=[\spec(R[x])/\GG_m]\to \X$, where the action is given by $x\mapsto t\cdot x$, the condition of $S$-competeness
is tested by families ${\operatorname{\ov{ST}_R}}:=[\spec R[x,y]/(xy-\pi)/\GG_m]\to \X$, where the action of $\GG_m$ is given by $t\cdot x=tx\,,\,t\cdot y=t^{-1}x$, and where $\pi$ is a local parameter for $R$. This is well explained in \cite[Section 2B]{Hein17}, (see also \cite{Lan75}).
\vskip 0.3 cm
One of the fundamental results  in \cite{AHLH18} is Corollary 6.18.
\vskip 0.3 cm

\be\label{ss-loc}
\text{\parbox{.85\textwidth}{\it Let $\X$ be an algebraic stack of finite type with affine diagonal. Assume that $\X$ admits a good moduli space $\pi: \X\to X$. Let $\X^{ss_\ell}\subset \X$ be the $\Theta$-semistable locus with respect to some class $\ell\in H^2(\X,\QQ)$. Assume that $\X^{ss_\ell}$ is quasi-compact and open in $\X$, then $\X$ admits a good moduli space, which is separated. This moduli space is proper if 
$\X^{ss_\ell}$ satisfies the existence part of the evaluative criterion for properness
and if the $\Theta$-stratification is well ordered.
}}
\ee

(the last condition will always be satisfied in the case of Bridgeland moduli spaces).
\vskip 0.3 cm

\subsection {Fixing a stability condition}\label{stab-cond}  As usual we denote by  $\D$ either the bounded derived 
of a K3 surface or the Kuznetsov component of a smooth cubic fourfold in $\PP^5$.
When no confusion will arise, we will denote with the same letter $X$ either a K3 surface or a
smooth cubic fourfold in $\PP^5$.

We  fix on $\D$ a numerical Bridgeland stability condition $\sigma$
with heart $\A\subset \D$,  and stability function $Z$:
\be\label{stab-cond}
\sigma=(\A, Z)\,,\qquad\A\subset \D.
\ee
We will continue to follow the notation in 
\cite{AHLH18}, \cite{Hal22}, and occasionally we will refer to
 \cite {\blm}, for the case $S=\spec k$.

\vskip 0.3 cm

\subsection {Moduli of complexes}\label{mod-cmpx}
\vskip 0.2 cm
If $X$ is a scheme, then  $\D^b_{pug}(X)\subset D^b(X)$  denotes the subcategory of 
{\it perfect, universally gluable}  objects  \cite[ Def 2.1.1, Def. 2.1.8]{Lie06} and
\cite[Def.  8.5] {\blm}. 
We consider the stack $\M_{pug}(\D)$, whose groupoids are given by
$$
\M_{pug}(\D)(T)=\{E\in\D^b_{pug}(X\times T)\,|\,E_t\in\D\}
$$
where $E_t$ is the derived restriction of $E$ to $X\times\{t\}$:
$$
E_t=\text{\bf L}\iota^*_t E\,,\qquad \iota_t:X=X\times\{t\}\to X\times T
$$
  Lieblich   \cite{Lie06},  proves that this is an Artin stack
locally of finite presentation \cite[Section 9]{\blm}.

\vskip 0.2 cm

\subsection {  The algebraic stack   $\M_v(\A)$}\label{mod-stack}

Let $\A$ be the heart of the Bridgeland stability condition in (\ref{stab-cond}). Consider the stack $\M(\A)\subset \M_{pug}(\D)$, whose groupoids are defined as follows. For every scheme of finite type $T$:
$$
\M(\A)(T)=\{E\in \M_{pug}(T)\,\,|\,\,E_t\in \A\}
$$
This definition was introduced by Abramovich and Polishchuk 
\cite[Definition 3.3.1]{AP06} , and also \cite[Definition 6.2.1]{Hal22},
 in a much broader generality, where $\A$ can be  the noetherian  heart of a
non-degenerate $t$-structure in the derived category of a smooth projective variety $X$.

{\it{\bf Remark: } We are departing slightly from Halpern-Leistner's notation \cite{Hal22},
where the symbol $\M$ is used instead of $\M(\A)$.}

In order to prove many important properties of this moduli stack, one needs to use a number of different but, a posteriori, equivalent incarnations of the stack
$\M(\A)$. All of them are thoroughly discussed in Sections 6.1,6.2 of \cite{Hal22}.
To get good results for $\M(\A)$, the heart $\A$ should satisfy noetherianity, but above all,
the so-called {\it generic flatness condition}. This last property was first envisaged by Artin and Zhang \cite[Section C5]{AZ01}, and both properties are pivotal in proving that:
 \cite [Proposition 6.2.6]{Hal22} and \cite[Theorem 3.20]{Tod08} for the case of Bridgeland stability.

\be\label{gms}
\text{\parbox{.85\textwidth}{\it Assuming noetherianity and
generic flatness,  $\M(A)$ is an open substack of  $\M_{pug}(\D)$ and hence an algebraic stack of finite type, with affine diagonal}}.
\ee
The fact that noetherianity holds for the heart of a numerical stability condition is due to Bridgeland \cite{Br07}. The fact that generic flatness holds  for the heart of a Bridgeland stability condition is due to Toda \cite[Lemma 4.7, Proposition 3.18 in the case of a K3 surface, and only for the main component, and the case of a Kuznetsov component is similar]{Tod08}.

The fundamental result  is \cite[Section 7.2]{AHLH18}:

\be\label{gms}
\text{\parbox{.85\textwidth}{\it Let $\A$ be as in (\ref{stab-cond}). Then $\M(\A)$ is $\Theta$-reductive and $S$-complete, and hence admits a good moduli space}}.
\ee

{\bf Remark:} In \cite{AHLH18} the properties of $\Theta$-reductivity, and $S$-completeness are proved  in Lemma 7.15 and Lemma 7.16, respectively, in grater generality, that  is for more general hearts. The result is achieved by using  another incarnation of the moduli-stack which they call $\M_\A$ \cite[Definition 7.8]{AHLH18}.
It should be noted that in that definition, a  more restrictive assumption is imposed on $\A$: that it should be {\it locally noetherian} (see e.g. \cite{Ro69}). This condition  is automatically satisfied by replacing a given  heart $\A\subset \D$ with $\A_{qc}=\operatorname{Ind}(\A)$. However, in     \cite[Example 7.20]{AHLH18}, the authors point out that, due to the important   \cite[Proposition 3.3.7]{P07}, their stack $\M_{\A_{qc}}$
is equivalent to Abramovich-Polishchuk's stack $\M(\A)$.
\vskip 0.3 cm
Let finally $v\in H^*_{alg}(X,\QQ)$ be a Mukai vector and let $\M_v(\A)$ be the stack defined by the groupoid:
\be\label{mva}
\M_v(\A)(T)=\{E\in \M(\A)(T)\,\,|\,\,\,\,v(E_t)=v\}
\ee
Then
\be\label{gmsv}
\text{\parbox{.85\textwidth}{\it Let $\A$ be as in (\ref{stab-cond}). Then $\M_v(\A)$ is $\Theta$-reductive and $S$-complete, and hence admits a good moduli space}}.
\ee

The algebraic stack $\M_v(\A)$  plays the role of the Quot-scheme 
in the classical picture.
In view of (\ref{ss-loc}), in order to introduce a technique that substitutes the non-available GIT in this algebraic stack, we  need an appropriate class
$\ell\in H^2(\M(\A),\QQ)$. Having done this we will be able to talk about the semistable locus 
\be\label{ssell}
\M^{ss_\ell}_v(\A)\subset \M_v(\A)
\ee
with respect to the $\Theta$-stratification induced by $\ell$, and by virtue of 
(\ref{ss-loc}), we will be able to conclude that $\M^{ss_\ell}_v(\A)$ has a good moduli space.
The stack $\M^{ss_\ell}_v(\A)$ turns out to be  exactly the replacement of the scheme $Q^{{(L_H\text{-}GIT})_{ss}}$, appearing in (\ref{two-stab}), Subsection \ref{classic},   of the classical picture. The cohomology class $\ell$, we are looking for, 
was introduced in a slightly different context by Bayer and Macr\`i in \cite{BM-pr14}.
Let us recall this construction.

\subsection {Bridgeland moduli stack $\M_\sigma(v,\D)$ and Bayer-Macr\`i's class}\label{bm-lb}
\label{BM-class}
Let $\sigma$ be as in \ref{stab-cond}.
Bayer and Macr\`i consider the stack  $\M_\sigma(v, \D)$ of {\it flat families of $\sigma$-semistable objects,
in $\D$, 
of class $v$}, that is the stack whose groupoids are given by
$$
\M_\sigma(v,\D)(T)=\{ E\in \M_{pug}(T)\,|\, E_t\,\,\text{is}\,\, \sigma\text{-semistable},\text{with}\,\, v(E_t)=v\,,\,\,\text{and fixed phase}\,\,\phi\}
$$
where $E_t$ is, as usual, the derived restriction of $E$ to $X\times\{t\}$. Following \cite{\blm}, Definition 21.11, we will omit the phase in the notation. 
The stack
$\M_\sigma(v,\D)$ is  algebraic  and it is  an open substack of $\M_{pug}(X)$ (see, e.g. \cite{\blm}, Lemma 21.12). As we may assume that $\phi\in(0,1]$, the stack $\M_\sigma(v,\D)$
is an open substack of $\M_v(A)$
 The stack $\M_\sigma(v,\D)$ is playing the role played by 
$Q_0^{(H\text{-}Gies)_{ss}}$ (defined in item {\bf 2)} Subsection \ref{classic})
in the classical picture.
In Subsection \ref{mod-stack-slope}  we will prove the fundamental equality
\be\label{git-slope}
\M_\sigma(v,\D)=\M^{ss_\ell}_v(\A)\qquad\qquad\subset \quad \M_v(\A)
\ee
which is the exact analogue of (\ref{two-stab}). Bayer and Macr\`i's construction goes as follows.
Set for simplicity
 $\M=\M_\sigma(v,\D)$. Cosider a universal object
 $\E$ over $\M\times X$ and set 
 $$
p_1: \M\times X\to\M\,,\quad p_2: \M\times X\to X	
$$
Recall (\ref{z-omega2}), and consider the cohomology class:
\be\label{ell}
\ell_\sigma=\ell:=\ch_1\left({p_1}_*\left(\E\otimes p_2^* \,\Im m\left(\frac{-\omega_\sigma}{Z(v)}\right)\right)\right)\in H^2(\M,\QQ)
\ee
(here we are modifying $\sigma$ so that it is defined over $\QQ[i]$).
This is the first Chern class of the line bundle $\frak{L}_\sigma$ introduced  in \cite{BM-pr14}. The way is now paved to find the class $\ell\in  H^2(\M_v(\A),\QQ)$: we just copy the above construction for the case 
of $\M_v(\A)$.

\subsection {$\Theta$-stability for $\M_v(\A)$}\label{mod-stack-git} Look at formula  (\ref{ell}),
substitute
  the stack $\M_\sigma(v,\D)$ and its universal family with the stack 
$\M_v(\A)$, and its universal family (which we shall call with the same name). This defines a class $\ell\in H^*(\M_v(\A), \QQ)$.
At this point a substitute for the unavailable  GIT on $\M_v(\A)$ presents itself under the guise of $\Theta$-stratification,
based on the class $\ell$. Now
(\ref{ss-loc}) tells us that 
\be\label{gms-mva}
\text{\parbox{.85\textwidth}{ $\M_v(\A)^{ss_\ell}$ {\it has a good moduli space} }}.
\ee
Let's see how to test $\Theta$-semistability  of a point $[E]\in |\M_v(\A)|$ by using the class $\ell$.
First of all, to simplify notation, we can always normalize the class $\ell$ in (\ref{ell})
by setting
\be\label{normaliz}
Z(v)=i
\ee
We can do this by acting with the group $\wt\GL^+(2, \RR)$ (see Section \ref{notation}).

$$
\text{\parbox{.85\textwidth}{\it    By definition a point $[E]\in |\M_v(\A)|$ belongs to $|\M_v(\A)^{ss_\ell}| $ if, for every filtration 
$f:\Theta\to \M_v(\A)$, with $f(1)=[E]$, we have $f^*\ell\geq 0$.
     }}.
$$

We must compute $f^*\ell\in K_0( \Theta\times X)$, and for this we follow Halpern-Leistner's \cite[Lemma 6.4.8]{Hal22}. Consider the 
diagram
$$
\xymatrix{\Theta\times X\ar[r]^{f\times 1}\ar[d]_\pi&\M\times X\ar[d]^{p_1}\,,\\
\Theta\ar[r]^f&\M
}\qquad f(1)=[E]
$$
Using (\ref{class-e}), and letting $q_2: \Theta\times X\to X$ be the projection,
we get (as usual, all functors should be interpreted in the derived sense)
$$
\aligned
f^*\ell&=f^*\left[\ch_1\left({p_1}_*\left(\E\otimes p_2^* \,\Im m\left(\frac{-\omega_\sigma}{i}\right)\right)\right)\right]\\
&=\ch_1\left[  {\pi}_*(f\times 1)^*\left(\E\otimes p_2^* \,\R e (\omega_\sigma)\right)\right]\\
&=\ch_1\left[  \left({\pi}_*\left(\E_f\otimes q_2^* \,\R e\left(\omega_\sigma\right)\right)\right)  \right]\\
&=\ch_1\left[  {\pi}_*\left(\left(\sum_ju^{-w_j}\operatorname{gr_j}E^\bullet\right)\otimes q_2^* \,\R e\left(\omega_\sigma\right)\right)  \right]\\
&=\ch_1\left[  \sum_ju^{-w_j}\R e\left({\pi}_*\left(\operatorname{gr_j}E^\bullet\otimes  \,\omega_\sigma\right)\right) \right]\\
&=\ch_1\left[  \sum_ju^{-w_j} \,\R e \left(\chi\left(\operatorname{gr_j}E^\bullet\otimes\omega_\sigma\right)\right)  \right]\\
&=\ch_1\left[  \sum_ju^{-w_j} \,\R e \left(\left<\operatorname{gr_j}E^\bullet\,,\,\omega_\sigma\right>\right)  \right]\\
&=\ch_1\left[  \sum_ju^{-w_j} \,\R e \left(Z(\operatorname{gr_j}E^\bullet\right)  \right]\\
&=  -\sum_j {w_j} \,\R e \left(Z(\operatorname{gr_j}E^\bullet)\right)  \\
\endaligned
$$

Writing
$$
\frac{Z(-)}{Z(v)}=\frac{Z(-)}{i}=\operatorname{rank}(-)+i\deg(-)=-i\R e(Z(-))+\Im m(Z(-))
$$
we finally get 
$$
f^*\ell=-\sum_jw_j\R e (Z(E_{w_j}))=\sum_jw_j\deg(E_{w_j})
$$
where we set $E_{w_j}=E^{w_j}/E^{w_{j+1}}$, and where, in the last expression, we allow some of 
the terms $E_j$ to be equal to $0$.

\subsection {Bridgeland stability vs $\Theta$-stability, and the equality: $\M_\sigma(v,\D)=\M^{ss_\ell}_v(\A)$}\label{mod-stack-slope}
{}

As we already mentioned, the equality we want to prove is the exact analogue 
of equality (\ref{two-stab}) in the classical picture. Unlike the classical case, this equality
is quite  straightforward in the stack setting. It is proved in \cite[Theorem 4.6.11]{Hal22}, and more explicitely in \cite[Lemma 7.22]{AHLH18}. This is what that  Lemma says:

\begin{lem}\label{bridg-git}
 A point $x=[E]\in |\M_v(\A)|$ is $\sigma$-unstable if and only if is $\Theta$-unstable with respect to $\ell$. That is: $\M_\sigma(v,\D)=\M^{ss_\ell}_v(\A)$.
\end{lem}

The proof is straightforward: we are always working under the assumption that $Z(v)=i$, and we first suppose that $F\subset E$ is a $\sigma$-destabilizing subobject.
Set $E_1:=F$, $E_0:=E$,   $E_i=0$, $i>1$,  $E_i=E$, $i<0$,      so that $w_1=1>w_0=0$, and $f^*\ell=-\R e (Z(F))>\R e (Z(E))=0$, proving the $\Theta$-unstability of $E$. Viceversa,
suppose $f^*\ell=-\sum_jw_j\R e (Z(E^{w_j}/E^{w_{j+1}}))>0$,   for some filtration
$E^\bullet$, with $E^{w_0}=E$, and $w_p>w_{p-1}>\cdots>w_0$. Now
$$
-\sum_jw_j\R e (Z(E^{w_j}/E^{w_{j+1}}))=-(w_1-w_0)\R e (Z(E^{w_1}))-\cdots-(w_p-w_{p-1})\R e (Z(E^{w_p}))>0
$$
so that $\R e (Z(E^{w_i}))<0$, for some $i$, proving the $\sigma$-unstabilty of $E$. Q.E.D.
\vskip 0.3 cm
From (\ref{gms-mva}) one concludes:

\begin{theorem}\label{gms-br}
$\M_\sigma(v,\D)$ has a good moduli space $M_\sigma(v,\D)$:
$$
\pi: \M_\sigma(v,\D)\to M_\sigma(v,\D),
$$
The algebraic space  $M_\sigma(v,\D)$ is called the Bridgeland moduli space of $\sigma$-semistable objects of given phase (and given Mukai vector $v$).
Moreover,
$M_\sigma(v,\D)$ is complete.
\end{theorem}
This statement is proved in \cite{AHLH18},  in various steps. First it is proved for their stack $\M_\A$  (go back to the Remark just after (\ref{gms})), and therefore, by their Example 7.20, for the Abramovich-Polishchuk's stack $\M(\A)$. By Corollary 6.12 on gets the properness of
$M^{ss_\ell}_v(A)$; the one of $M_\sigma(v,\D)$ follows from (\ref{git-slope}).

\vskip0.3cm
\section{Kuranishi families}\label{Kuranishi-f}\label{kura-fa}

\subsection{Bases of Kuranishi families}\label{base-kura}

We keep the notation introduced in   the preceding section.
Let $[F]\in M_\sigma(v,\D)$ be a point corresponding to a $\sigma$-semistable object $F$ in the heart $\A$ of a numerical stability condition $\sigma=(\A, Z)\in\Stab^\dagger(\D)$.  As in (\ref{gms-br}), we  let
$x\in\M_\sigma(v,\D)$ be such that $\pi(x)=[F]$, so that $G_x=\Aut(F)$.
Let us emphasize
 that it is thanks to the very  construction of the stacks
$\M_v(\A)$, and $\M_\sigma(v,\D)$, originally conceived 
by Lieblich, and Abramovich-Polishchuk and mentioned in 
Subsections \ref{mod-cmpx}, and \ref{mod-stack},
that we can speak about infinitesimal deformations of $F$, or more generally about families of objects in $\A$ parametrized by a scheme $T$.
\vskip 0.2 cm 

\vskip 0.2 cm 
Exactly as in the classical case of sheaves, one can set up the formalism 
of Artin-Schlessinger deformation theory. The upshot of the obstruction theory is the construction of a formal map, called   {\it a Kuranishi map for $[F]$ }:
\be\label{kura}
\kappa=\kappa_2+\kappa_3+\cdots \quad: \widehat{\Ext^1(F,F)}\longrightarrow \Ext^2(F,F)_0
\ee
(where the space on the right hand side is the kernel of the trace map from $\Ext^2(F,F)$ to $H^2(\cO_S)$), having the property that the formal scheme $D_\kappa:=\kappa^{-1}(0)$ parametrizes formal deformations of $F$ in the sense that $\kappa^{-1}(0)$ is the base for a {\it formal  versal deformation}  of $F$, called a {\it formal Kurranishi family for $[F]$}:

\be\label{kuran-fam}
\xymatrix{\widehat \F_\kappa\ar[d]^{\widehat \phi_\kappa}\\
S\times \kappa^{-1}(0)
}
\ee
While the higher order terms of a Kuranishi map are quite mysterious  and not unique since they depend on several choices, the second order term is transparent, and  independent from the specific choice of a Kuranishi map $\kappa$:

$$
\aligned
\kappa_2: \widehat{\Ext^1(F,F)}&\longrightarrow\Ext^2(F,F)_0\\
&e\quad\mapsto\quad e\cup e
\endaligned
$$
In other words if
$$
\mu: \Ext^1(F,F)\longrightarrow\Ext^2(F,F)_0
$$
is the cup product map, then $\kappa_2^{-1}(0)=\widehat{\mu^{-1}(0)}$.
The group  $G_x=\Aut(F)$ acts naturally on the domain and the codomain
of the Kuranishi map (\ref{kura}). It is a fundamental theorem by Rim [R], that there exists a $G_x$-equivariant Kuranishi map $h$ and a $G$-equivariant formal Kuranishi family $(\wh\F_h,\wh\phi_h)$, which is unique up to a unique $G$-equivariant isomorphism.

\vskip 0.2 cm
In the next section, we will show that the differential graded Lie algebra (dgla) $R\Hom(F,F)$ satisfies the {\it formality property}, meaning that there is {\it roof}  of quasi-isomorphisms between $R\Hom(F,F)$ and its cohomology algebra $\Ext^*(F,F)$ (regarded as  a dgla with trivial differential).
One of the consequences of  formality is the existence of an isomorphism 
\be\label{quadr}
\kappa^{-1}(0)\cong\kappa_2^{-1}(0)=\widehat{\mu^{-1}(0)}
\ee
It follows that the formal deformation space $\kappa^{-1}(0)$ is isomorphic to a  complete intersection of quadrics in $\widehat{\Ext^1(F,F)}$. More is true  \cite[Theorem 4.20] {Bmm19}: if 
$\kappa$ is a $G_x$-equivariant Kuranishi map then
one can choose the isomorphism (\ref{quadr}) to be $G_x$-equivariant (notice that $\kappa_2$ is trivially $G_x$-equivariant). It follows that
$\kappa_2^{-1}(0)$ is the base of a $G_x$-equivariant Kuranishi family.
\vskip 0.2 cm
Using the language of \cite[Section 21.2]{AHR20}, we see that the base $\kappa^{-1}(0)$ of a Kuranishi family $\kappa$
is isomorphic to a {\it miniversal deformation space $\widehat{\Def}(x)$ of $x$}:
\be\label{form-iso1}
\kappa^{-1}(0)\cong\widehat{\Def}(x)
\ee

More precisely, denoting by $\widehat \M_\sigma(v, \D)_x$ a formal neighbourhood of  $[F]\in \M_\sigma(v, \D)$, the formal deformation space  
$\kappa^{-1}(0)$
is a formal affine scheme and the natural map $\kappa^{-1}(0)\to\widehat \M_\sigma(v, \D)_x$ defined by the versal family $(\wh\F_\kappa,\wh\phi_\kappa)$, is formally smooth  because
of the  versality of the family $(\wh\F_\kappa,\wh\phi_\kappa)$. Following again the notation of \cite{AHR20}, we have a formal isomorphism:
$$
[\kappa^{-1}(0)/G_x]\cong {\widehat \M_\sigma(v, \D)_x}
$$

\subsection{Comparing with \'etale slices}\label{kur-vs-slic}

Following  \cite{AS18}, we next compare bases of Kuranishi families with \'etale slices. As in the preceding section, let $U$ be an \'etale slice
at $x\in  \M_v(\sigma)$,  and let $y\in U$ be a point mapping to $x$.
From the construction of the \'etale slice in \cite{AHR20}
there is a formal isomorphism 
$$
\widehat U_y\cong\widehat{\Def}(x)
$$
and therefore, by (\ref{form-iso1}), a formal isomorphism 
\be\label{form-iso2}
\widehat U_y\cong \kappa^{-1}(0)
\ee
By formality,  we can assume that $\kappa^{-1}(0)$ is $G_x$-equivariantly  isomorphic to $\kappa_2^{-1}(0)$. We can now proceed as in \cite{AS18}. Both the right hand side and the left hand side of (\ref{form-iso1}) are acted on by $G_x$, moreover, the isomorphism between them
induces a $G_x$-equivariant isomorphism at the level of tangent spaces (mediated by 
$G_x$-equivariant isomorphisms between each of these tangent spaces  and $T_x \M_v(\sigma)$). We can then apply Starr's result \cite[Proposition 4.7]{AS18} to change, if necessary, 
the isomorphism (\ref{form-iso1}), and make it $G_x$-equivariant. Now we use Bierstone and Milman result  \cite[ Proposition 4.7 ]{AS18} and we get a $G_x$-equivariant analytic local isomorphism
$$
(U,y)\cong (\nu^{-1}(0), 0)
$$
inducing the identity on tangent spaces $T_y(U)=\Ext^1(F,F)=T_0(\nu^{-1}(0))$. 
Via this isomorphism a $G_x$-equivariant analytic neighbourhood  of $0$ in $\nu^{-1}(0)$
is the base for a $G_x$-linearized deformation of $F$.
In conclusion we  get the following analogue 
of   \cite[Proposition 4.4 and Proposition 7.4]{AS18}.

\begin{prop}\label{et-kuran} Let  $v$ be an element in $K_{\num}(\D)$, and  let $\sigma=(\A, Z)\in\Stab^{\dag}(\D)$ be a full  numerical stability condition. Let $x=[F]\in\M_\sigma(v,\D)$ be a point corresponding to a polystable object $F\in \A$. Let $G=G_x$ be the stabilizer of $x$ in $\M_\sigma(v)$, and let $[U/G]\rightarrow \M_\sigma(v,D)$ be a slice \'etale at $x=[F]$.
Let $y\in U$ be a point mapping to $x$ under the morphism $U\rightarrow \M_\sigma(v)$.
Let $\mu: \Ext^1(F,F)\longrightarrow\Ext^2(F,F)_0$ be the cup product map.
Then
 there are $G$-invariant, saturated analytic neighbourhoods  $\U\subset U$ of $y$, and 
 $\V\subset \mu^{-1}(0)$ of $0$, and a $G$-equivariant pointed isomorphism 
\be\label{utov}
 \phi:(\U, x)\to (\V,0)
\ee 
 Furthermore there is 
 a $G$- linearized deformation of $F$ parametrized by the $G$-invariant analytic neighbourhood  $\U$.
 
 \end{prop}
 
 \subsection{ Formality of the Kuranishi family: the K3 case.}\label{formality}
 In this section we set $\D=\D^b(X)$, for a K3 surface $X$.
Let $[F]\in M_\sigma(v, \D)$ be a point corresponding to a $\sigma$-polystable object $F$ in the heart $\A$ of a numerical stability condition $\sigma=(\A, Z)$.  As above we  let
$x\in\M_\sigma(v,\D)$ be such that $\pi(x)=[F]$, so that $G_x=\Aut(F)$. We write
$$
F=\oplus_{i=1}^sF_i\otimes V_i
$$
where each $F_i$ is $\sigma$-stable, so that
$$
G:=G_x=\prod_{i=1}^s\GL (V_i)
$$
the result we want to prove is the following theorem proved in \cite{Bmm21} for the case where $F$ is a Gieseker semistable sheaf on a K3 surface.

\begin{theorem}\label{formalKu}  Let $F\in \A$ be $\sigma$-semistable. Then the DG-algebra
$\Rhom(F,F)$ is formal.
\end{theorem}

\begin{proof} Much  of  this proof can be obtained by following verbatim the proof of Theorem 5.1 in  \cite{Bmm21}. The first step is to substitute 
each $F_i$ with a  finite, locally free resolution $E_i^{\tb}\to F_i$. We then set
$E^{\tb}=\oplus_{i=1}^sE_i^{\tb}\otimes V_i\overset{q.i.}\cong F$. The action of $G$
on $F$, defined solely in terms of the $V_i$'s is authomatically transferred to an action on $E^\tb$, and the preceding quasi-isomorphism is trivially $G$-equivariant.
We then have
$$
\Rhom(F,F)=\Rhom(E^\tb,E^\tb)=\RG(\hom(E^\tb,E^\tb))
$$
the last identification being given by
the local freeness of the $E_i^\tb$'s. Denote by $\A_X^{0,\tb}(G^\tb)$
the total complex of sheaves of $(0,\tb)$-forms 
with coefficient in a complex $G^\tb$. Tensoring  the quasi-isomorphism $\A_X^{0,\tb}(\cO_X)\cong \cO_X$, by $\hom(E^\tb,E^\tb)$,  we get, again by local freeness:
\be\label{Rhom}
\Rhom(F,F)=\Rhom(E^\tb, E^\tb)=\RG(\hom(E^\tb,E^\tb))=\RG(\A_X^{0,\tb}(\hom(E^\tb,E^\tb)))
\,,
\ee
in $\D^b(\{\text{pt}\})$.
Set 
$$
A_X^{0,\tb}(\hom(E^\tb,E^\tb)):=\RG(\A_X^{0,\tb}(\hom(E^\tb,E^\tb)))\quad\text{and}\quad L^\tb:=A_X^{0,\tb}(\hom(E^\tb,E^\tb))
$$
The $(p,q )$-component of the differential of ${L^\tb}$:
$$
d_{L}: \,A^{0,p}(\hom^q(E^\tb,E^\tb))\longrightarrow A^{0,p+1}(\hom^q(E^\tb,E^\tb))\oplus A^{0,p}(\hom^{q+1}(E^\tb,E^\tb))
$$
is given locally by
\be\label{diff-L}
\phi=\omega\otimes f\quad\mapsto \quad d_{L} (\phi)\quad=\quad\ov\partial \omega\otimes f\quad\pm\quad\omega\otimes d_{\hom(E^\tb, E^\tb)}(f)
\ee

As above, $L^\tb$
is naturally equipped with a $G$-action solely involving the $V_i$'s. As such, the  $G$-action  does not interfere with the $\ov\partial$-differential coming from $\A_X^{0,\tb}$, and therefore commutes with the differential $d_{L}$. Thus,  the quasi-isomorphism $\Rhom(F,F)\cong L^\tb$ is  $G$-equivariant.  We also have
$$
Z^0(L^\tb)=\Hom_{\Kom(X)}(E^\tb, E^\tb)=\oplus\Hom_{\Kom(X)}(E_i^\tb, E_i^\tb)\otimes\frak{gl}(V_i)
$$
Recall that each $F_i$ is stable, and hence simple, so that, given an element

$$
f_i=(\dots, f_i^{-1}, f_i^0)\in\Hom_{\Kom(X)}(E^\tb_i, E^\tb_i)
$$

and writing $\ov {f_i^0}: F_i=E^0_i/dE^{-1}_i\to E^0_i/dE^{-1}_i=F_i$, we have   $\ov {f_i^0}=c_i\cdot 1_{F_i}$, with $c_i\in \CC$.
Now consider the natural injective 
homomorphism 
\be\label{fi}
\aligned
 \Lie&(G)=\oplus_{i=1}^s\frak{gl}(V_i)\longrightarrow Z^0(L^\tb)\\
&h=(h_1,\dots, h_s)\,\,\longmapsto\,\,(1_{E_1^\tb}\otimes h_1,\dots,1_{E_s^\tb}\otimes h_s)
\endaligned
\ee
and the natural  projection
$$
\aligned
&Z^0(L^\tb)\longrightarrow H^0(L^\tb)=\Hom_{\D}(F,F)=\Lie(G)\\
(f_1\otimes h_1,\dots, f_s&\otimes h_s)\,\,\longmapsto\,\,1_{F_1}\otimes c_1h_1,\dots 1_{F_s}\otimes c_sh_s)
\endaligned
$$
Their composition is the identity. We are now in the position to use Corollary 4.5 in \cite{Bmm21}.
The fact that  $L^\tb$ is a quasi-cyclic DG-Lie algebra of degree 2, with $H^i(L^\tb)=0$ for $i\neq0,1,2$, is proved exactly as in the beginning of   \cite[Section 5]{Bmm21}. Regarding item (1)
in Corollary 4.5 (loc.cit.), we can take $H^0=\Lie(G)$ embedded in $Z^0(L^\tb)$ via (\ref{fi}). Regarding item (2), one proceeds exactly as in the last three paragraphs of the proof of Theorem
5.1 (loc.cit.). In conclusion  $L^\tb$ is formal 
\end{proof}

\begin{rem}
Essentially the same proof of the theorem above appears also in \cite{CPZ21}. The only difference is the choice of the $G$-equivariant locally free resolution: ours is easier to construct, but only works for polystable objects (more precisely, for objects that are direct sum of simple objects which have no homomorphisms between non isomorphic summands).
\end{rem}

\begin{rem}\label{simple} In the proof of the preceding theorem, we never used the stability of the $F_i$'s, but only the fact that 
$\dim\Hom(F_i,F_j)=\delta_{i,j}$.
\end{rem}

\begin{rem}\label{twist} The  same proof works for the case in which 
$\D=\D^b(X,\alpha)$ is the derived category of a K3 surface $X$ twisted by a Brauer class $\alpha$. For this one should notice that finite, locally free resolutions exist also in the twisted case \cite[p.908]{HL10} \cite[Lemma 2.1.4]{Cal00} and that 
the $\hom$-complex is untwisted.
\end{rem}

\subsection{ Formality of the Kuranishi family for the Kuznetsov component}\label{formality-K}

Let $X$ be a cubic fourfold.
Let $\Ku(X)$ be the Kuznetsov component of $\D^b(X)$.
 For the definition of the lattice $\Lambda:=\wt H^{1,1}(\Ku(X,\ZZ))$ we refer the reader to \cite[Section 3.4]{MS17}.
Let $v\in\Lambda$, and $\sigma\in \Stab_\Lambda^\dagger(\Ku(X))$. Finally let $F=\oplus_{i=1}^sF_i\otimes V_i\in \Ku(X)$ be a $\sigma$-polystable
object with $v(F)=v$, where $F_i$ is $\sigma$-stable, with $v(F_i)=v_i$, and $\dim_\CC V_i=n_i$. We will prove the following theorem.

\begin{theorem}\label{form-Kuz} The dg-algebra $R\Hom(F,F)$ is formal.
\end{theorem}

Before proving this theorem we need
a preliminary result. Recall  that $\Ku(X)$ is said to be {\it geometric} if 
$\Ku(X)\cong\D^b(X,\alpha)$, for some K3 surface $S$ and some Brauer class  $\alpha$.

\begin{prop}\label{density} Let $\X\to T$ be a non-isotrivial family of cubic fourfolds. Then the subset $\{t\in T \,\,|\,\,\Ku(X_t)\,\,\text{is geometric}\}$ is dense in the analytic topology.
\end{prop}

\proof By \cite[Proposition 33.1]{BLM${}^+$19}, $\Ku(X_t)$ is geometric if and only if there exists $v\in \Lambda_t:=\wt H^{1,1}(\Ku(X_t,\ZZ)$ such that $v^2=0$.
On the other hand,  by the Hasse-Minkowski theorem, the semi-definite quadratic form on $\Lambda_t$ represents zero as soon as $\rho_t:=\rk \Lambda_t\geq 5$. It is therefore enough to prove that
$$
T_{\geq5}:=\{t\in T \,\,|\,\,\rho_t\geq5\,\}
$$
is dense in $T$. But this follows immediately from the following fact.

{\bf Fact}.(Green) \cite[Proposition 17.20]{V04} {\it Let $\X\to T$ be as above.
Let $\ov \rho=\min_{t\in T}\{\rho_t\}$ then the Noether-Lefschetz locus
$NL=\{t\in T\,\,|\,\,\rho_t>\ov\rho\}$ is dense in $T$}.

\vskip 0.2 cm
An immediate consequence is that for every fixed $\rho>\ov\rho$, the subset
$\{t\in T\,\,|\,\,\rho_t>\rho\}$ is dense in $T$; in fact it is a dense countable union 
of closed proper components.
\endproof

\proof (of Theorem \ref{form-Kuz}) If $\D$ is equivalent to the derived category $\D^b(S,\alpha)$  of $\alpha$-twisted sheaves 
on a K3 surface $S$ (for some Brauer class $\alpha$), then we are done by using 
Theorem \ref{formalKu} and \cite[Proposition 1.4]{BZ19}.
Let $\C$ denote the moduli space of cubic fourfolds, and let $\C_{v_1,\dots v_s}$ be the locus where the degree four part of the  vectors $v_1,\dots v_s$ stays algebraic. We may assume that $\dim \C_{v_1,\dots v_s}>0$.
In fact, by the same reasoning (using Hasse Minkowski) as in the proof of the preceding Proposition,
we may assume that  $\rk\langle v_1,\dots v_s\rangle\leq 4$, so that
$\codim_\C \,\,\C_{v_1,\dots v_s}\geq 4$.
We can then take a non constant family 
$f:\X\longrightarrow T$ parametrized by a smoooth affine curve $T\subset  \C_{v_1,\dots v_s}$,  such that $\X_{t_0}=X$:
the cubic fourfold we started with. By construction, the vectors $v_i$ belong to the local system
$K_{num}(\operatorname{Ku}(\X/T))$. We claim that, up to a base change, we can assume to have families of objects
$\F_i\in \D(\X/T)$, with $\F_{i,0}=F_i$, and such that $\F_{i,t}$ is simple, and $\Hom(\F_{i,t}, \F_{j,t})=0$, for $i\neq j$.
But this is Theorem 1.4 in \cite{Per22}, plus semicontinuity. By what we said, there is a dense set of points  $t\in T$
such that $\D_{t}=\operatorname{Ku}(X_{t})=D^b(S_t,\alpha_t)$, for some twisted K3 surfaces $(S_t,\alpha_t)$.
Set $T=\spec R$. 
By construction the family $f:\X\longrightarrow T$ is  equipped with a relative  object 
$$
\F=\oplus_{i=1}^{s}\F_i\otimes V_i\quad\in\quad \Ku(\X/T)\quad\subset \quad\D^b_{\text{perfect}}(\X/T)\,,\qquad 
$$
In order to prove the formality of $R\Hom(F,F)$, we want to find a good $R$-algebra representative  for 
$$
Rf_*\hom(\F, \F)\quad\in\quad \D^b(\spec (R))
$$
and then use Lunts's result  \cite{Lunts08}, which  tells us that, under suitable hypotheses, the locus of points $t\in T$
for which $R\Hom(\F_t,\F_t)$ is formal is a closed set.
Passing from an $R$-module $M$ to the corresponding sheaf $\wt M$ gives the 
identification $\D^b(\spec (R))=\D^b(R\text{-Modules})$.
As $T$ is affine, by virtue of Lieblich's   \cite[Corollary 2.1.7]{Lie06}, $T$-perfect complexes on $\X$ are in 
fact strictly  perfect, so that
$$
\D^b(\X/T)=\{\F\in \D^b(\X)\,\,|\,\, \F\overset{q.i.}\cong \E^\ftb=\oplus \E^j\,,\text{bdd complex, with}\,\, \E^j\,,\,\, T\text{-flat} \,    \}
$$
For each $i=1,\dots, s$, choose a quasi-isomorphism $\E^\ftb_i\cong\F_i$, where $\E^\ftb_i$
is a bounded complex and the $\E^j_i$'s are flat over $T$. Let $\U=\{U_i\}_{i\in I}$ be an affine cover of $\X$, with $U_i=\spec B_i$. For $\alpha=\{\alpha_{i_0},\dots,\alpha_{i_n}\}$, with $\alpha_{i_k}\in I$, set $|\alpha |=n$, and
$U_\alpha=U_{\alpha_{i_0}}\cap\cdots\cap U_{\alpha_{i_n}}=\spec B_\alpha$. As $R$ is a PID, the flat $R$-module $B_\alpha$ is  projective. 
 Consider the complex
$\hom(\E^\ftb, \E^\ftb)=\oplus_i\hom^i(\E^\ftb, \E^\ftb)=\oplus_{i,j}\hom(\E^i, \E^j)$, and its 
$\check{\text{C}}$ech double complex
$$
\underset {p,q}\oplus\,\C^p(\U,\hom^q(\E^\ftb, \E^\ftb))
=\underset {p,q}\oplus\,\prod_{|\alpha|=p}\hom^q({\E^\ftb}_{|U_\alpha}, {\E^\ftb}_{|U_\alpha})
$$
The total complex $\wt\A\in \D^b(\spec(R))$ of this double complex, is flat over $\spec(R)$, and has a dg-structure,
with differential  given by $\delta\pm d$, where $\delta$ is the $\check{\text{C}}$ech differential, and $d$ is the differential of the complex $\hom(\E^\ftb, \E^\ftb)$. The complex $\wt\A$ represents
$Rf_*\hom(\F, \F)\in\D^b(\spec (R))$. 
On the other hand, the total complex $\A\in \D^b(R\text{-Modules}) $ of the corresponding $\check{\text{C}}$ech double complex of $R$-modules
$$
\underset {p,q}\oplus\,C^p(\U,\hom^q(\E^\ftb, \E^\ftb))
=\underset {p,q}\oplus\,\prod_{|\alpha|=p}\Hom^q({\E^\ftb}(U_\alpha), {\E^\ftb}(U_\alpha))
$$
is a dg-algebra, which, with its $R$-module structure, represents, (that is, it  is q.i. to) the object $Rf_*\Hom(\E^\ftb, \E^\ftb)\in  \D^b(R\text{-Modules})$.
In conclusion: $\wt\A$ represents $Rf_*\hom(\F, \F)\in\D^b(\spec (R))$, and $\A=R\Gamma\wt A$
represents $R\Gamma Rf_*\hom(\F, \F)=  Rf_*\Hom(\E^\ftb, \E^\ftb)\in  \D^b(R\text{-Modules})$.
From this point of view:
$$
[{R\Gamma Rf_*\Hom(\F, \F)}]^{\,\,\wt{}}=Rf_*\hom(\F, \F)
$$

Refining the cover $\U$, if necessary, we may assume that $\Hom^q({\E^i}(U_\alpha), {\E^j}(U_\alpha))$ is projective over $T$. This makes the dg-algebra  $\A$ a {\it cofibrant $R$-module}.
Moreover, base change in cohomology gives
\be\label{request}
A\otimes_Rk(t)=\wt\A_t=L\iota_t^*Rf_*\hom(\F, \F)=\Rhom(\E^\ftb_t, \E^\ftb_t))\in \D^b(\text{pt})
\ee
Where $\wt\A_t=(\wt\A)_{\spec(k(t))}$. Finally
$$
H^i(\wt A)=\Gamma(T, \H^i(Rf_*\hom (\E^\ftb, \E^\ftb)))
$$
where $\H^i(Rf_*\hom (\E^\ftb, \E^\ftb):=\ext^i_f(\E^\ftb, \E^\ftb)$.
It follows  that $\A$ verifies the hypotheses of Corollary 6.10 (b) and Remark 6.11 in \cite{Lunts08}.
Thus the locus  of  $t\in\spec R$, where $\wt\A_t$ is formal is closed. On the other hand, by the previous Proposition, there is a dense set in $\spec R$, where $\Ku(X_t)\cong\D^b(S_t,\alpha_t)$, where $(S_t,\alpha_t)$ is a twisted K3 surface. The conclusion comes from Theorem \ref{formalKu}.\endproof


\section{ Quivers}
Regarding  the notation and the terminology for quiver varieties we refer to \cite[section 5, pp. 673-680]{AS18}.
Here we  recall, very briefly, a few fundamental facts.
\subsection{The moment map and semi-stability for quiver representations}

Consider a quiver $Q=(I,E)$, and its double $\ov Q=(I, E \sqcup E^{op})$. 
Having fixed a {\it dimension vector}
$$
\mathbf n=(n_1,\dots, n_s)\in \ZZ_{\ge 0}^s
$$
and vector spaces $V_i$, $i=1,\dots,s$, with
$$
\dim V_i=n_i,
$$
one defines the vector space of  {\it $\mathbf n$-dimensional representations of} $Q$
by setting
$$
\rep(Q, \mathbf n):=\underset{e\in E}\bigoplus \Hom (V_{s(e)}, V_{t(e)})
$$
Where $s(e)$, and  $t(e)$ are the source and the tail of the edge $e$.
There is 
 an identification
$$
\rep(\ov Q, \mathbf n)=\rep(Q, \mathbf n)\oplus \rep( Q, \mathbf n)^\vee
$$
and thus a natural symplectic form on $\rep(\ov Q, \mathbf n)$. The group
$$
G:=G(\mathbf n):=\prod_{i=1}^s\GL(n_i)
$$
acts on $\rep(\ov Q, \mathbf n)$ in a natural way via conjugation, and
since this action  respects the symplectic form there is a {\it moment map} given by
\be \label{moment map}
\aligned
\mu_{\mathbf n}: \rep(\ov Q, \mathbf n)& \longrightarrow \gl(\mathbf n)\cong \gl(\mathbf n)^\vee\\
(x,y^\vee) & \longmapsto\mu(x,y^\vee)=\sum_{e\in E}[x_e, y_e^\vee]
\endaligned
\ee
 
Following Marsden-Weinstein symplectic reduction,  one is interested in the GIT quotient
$$
\mathfrak M_{\chi}(\mathbf n):=\mu_{\mathbf n}^{-1}(0)\sslash_{\chi} G
$$
where
$\chi: G\to \CC^\times$ is a rational character. 
\vskip 0.2 cm
{\bf Brief reminder.}
Here we 
recall that, if $X=\spec A$, then
$X\sslash_\chi G=\operatorname{Proj}\left(\underset{n\geq 0}\oplus A_n\right)$
where 
$A_n=\{f\in A\,\,|\,\,(g\cdot f)(x)=\chi(g)^nf(x)\}$. Points  $x\in X$
for which there exists $f\in A_n$
with $f(x)\neq0$ are said to be {\it $\chi$-semistable}, and they form a subset $X_\chi^{ss}\subset 
X$. Points in $X_\chi^{ss}$ with closed orbit and finite stabilizers are called {\it$\chi$-stable}, and they form a subset $X_\chi^{s}\subset 
X_\chi^{ss}$.
Two $\chi$-semistable points are said to be {\it$S_\chi$-equivalent}  if the closure of their orbits meet in $X_\chi^{ss}$. 
The projection \be\label{GIT-quot}
X_\chi^{ss}\longrightarrow X\sslash_\chi G
\ee
establishes a one-to-one correspondence between 
the set of closed points in  $X_\chi^{ss}\longrightarrow X\sslash_\chi G$ and the set of
$S_\chi$-equivalence classes in $X_\chi^{ss}$. The natural map
$$
\xi:  X\sslash_\chi G\longrightarrow X\sslash G=\spec(A^G)
$$
is often a desingularization. 

\vskip 0.2 cm
We are interested in desingularization maps
$$
\mathfrak M_{\chi}(\mathbf n):=\mu_{\mathbf n}^{-1}(0)\sslash_{\chi} G \longrightarrow \mu_{\mathbf n}^{-1}(0)\sslash G
=\mathfrak M_{0}(\mathbf n)
$$

Moreover we have $\dim\mathfrak M_{\chi}(\mathbf n)=d(\mathbf n)+2$, where

\be\label{dn}
d( \mathbf n)= {}^t\mathbf  n D  \ \mathbf n
\ee
and $D$ is the negative of the Cartan matrix associated to $Q$,
\cite[p. 674]{AS18}. In Crawley-Boevey's notation \cite{C-B01}:

\be\label{dim-quiv}
\dim\mathfrak M_{\chi}(\mathbf n)=d( \mathbf n)+2=2p( \mathbf n)
\ee
\vskip 0.2 cm
Rational characters 
$\chi: G(\mathbf n)\to \CC^\times$
are in a one-to-one correspondence with vectors
$
\theta=(\theta_1,\dots,\theta_s)\in \ZZ^s
$
via the formula 
$$
\chi_{\theta}(g)=\prod\det(g_i)^{\theta_i}
$$
where $g=(g_1,\dots,g_s)\in G=G(\mathbf n)$. 
\vskip 0.2 cm 
It is an important fact that stability and semi-stability may be expressed in terms of slope.
Consider, in $\ZZ^s$, the orthogonal complement of the dimension vector:
$
\mathbf n^\perp\subset \ZZ^s,
$
and set
$$
\W_\mathbf n= \mathbf n^\perp\otimes \QQ
$$

consider a point (a character) $\theta\in\W_\mathbf n$.
Let $V= \oplus V_i$ be an $ \mathbf n$-dimensional representation of  $\ov Q$. For any sub-representation 
$
W=\oplus_{i\in I} W_i,\,(W_i\subset V_i)
$
 the $\theta$-slope of $W$ is defined by setting 
$$
{\slope}_\theta(W)=\frac{\theta\cdot\dim W}{\sum \dim W_i }=\frac{\sum_{i=1}^s\theta_i\dim W_i}{\sum \dim W_i }
$$
so that, in particular, $\slope_\theta(V)=0$. A non-zero representation $V$ is said to be {\it $\theta$-semistable} if, for every sub-representation $W$ of $V$, we have $\slope_\theta(W)\leq 0$ and is said to be {\it $\theta$-stable}, if  the strict equality holds for every non-zero, proper sub-representation. A $\theta$-semistable representation admits a 
Jordan-H\"older filtration, and
 two representations are said to be  $S_\theta$-equivalent, if their associated, graded, Jordan-H\"older objects are isomorphic.
\vskip 0.2 cm 
It is a theorem due to King \cite[ p. 677]{AS18} that a representation
$V\in\Rep(\ov Q, \mathbf n)$, is $\theta$-semistable if and only if it is  $\chi_\theta$-semistable
(and the same is true for stability), and that $S_\theta$-equivalency coincides with  $S_{\chi_\theta}$-equivalency.
\vskip 0.2 cm 
\subsection{ Walls and chambers decomposition of $\W_\mathbf n$} For this we fix 
a {\it "positive root"}, that is an element $\alpha \in \ZZ^s_+$ such that $\alpha_i\leq n_i$, for all $i\in I$ , and having the further property  that
the
subgraph of $Q$
consisting of those vertices $i$ for which $\alpha_i\neq0$  and all the edges joining these vertices is connected \cite{AS18}, \cite{Kac80}, \cite{Kac82}. 
If $\alpha$ positive root, then $d(\alpha)+2\geq0$, so that, 
in the quiver world, positive roots are the equivalent of  Mukai vectors $v$ with $v^2+2\geq0$.
The {\it  (potential) wall associated to a positive root $\alpha$} is defined by
\be \label{W alpha}
\W_\alpha(\mathbf n)=\{\theta\in \W_\mathbf n\,\,|\,\,\theta\cdot\alpha=0\}
\ee
A potential wall is a bona fide wall if 
there actually exists a strictly $\theta$-semistable representation $V\in\Rep(\ov Q, \mathbf n)$, with an $\alpha$-dimensional sub-representation $V' \subset V$, with $\slope_\theta(V')=\slope_\theta(V)(=0)$. By definition, the  {\it chambers of  $\W_\mathbf n$} are the connected components of the complement of the walls, in a given wall, and   {\it the maximal dimensional chambers of  $\W_\mathbf n$} are the connected components of the complement of the walls. The set of semistable representation does not vary in a given chamber \cite[Proposition 5.9, p. 679]{AS18}.

\subsection{The Ext-quiver} We now connect the preceding picture with Proposition \ref{et-kuran}.
Let us fix 
a (numerical, full) stability condition $\sigma _0=(\A_0, Z_0)\in\Stab^\dagger(\D)$.
\vskip 0.3 cm
We are working   under
the formality assumption for the Kuranishi family: Subsections \ref{kuran-fam}, \ref{formality-K}.

\vskip 0.3 cm

We fix a $\sigma_0$-polystable object in $F\in \A$: 
\be\label{effe}
F=\oplus_{i=1}^sF_i\otimes V_i
\ee
With $F_i$ a $\sigma_0$-stable object with phase equal to the phase of $F$:  $\phi(F_i)=\phi(F)\in (0,1]$, $i=1,\dots,s$.
We set 
\be\label{notationF}
\aligned
v=&v(F)\,,\quad v_i=v(F_i)\,,\qquad v=\sum_{i=1}^sn_iv_i\,,\\
n_i&=\dim V_i\,\,,\quad i=1,\dots, s\,,\qquad {\bf n}=(n_1,\dots,n_s)
\endaligned
\ee
$$
G:=G({\bf n}):=\prod_{i=1}^sGL(n_i)\cong 
\Aut(F)
$$
Exactly as in  \cite[Proposition 6.1]{AS18}, one proves that
there exist a quiver $Q=Q(F)$ and 
$G(\mathbf{n})$-equivariant isomorphisms
\[
\Rep(\overline Q, \mathbf{n}) \cong \Ext^1(F,F), \quad  \mathfrak{gl}(\mathbf{n})^\vee \cong \Ext^2(F,F),
\]
such that, via these isomorphisms, the quadratic part  of the Kuranishi map for $F$
\[
\mu=k_2: \Ext^1(F, F) \to \Ext^2(F,F),
\]
can be identified with the moment map:
\be \label{moment map}
\aligned
\mu_{\mathbf n}: \Rep(\ov Q, \mathbf n)& \longrightarrow \gl(\mathbf n)\cong \gl(\mathbf n)^\vee\\
(x,y^\vee) & \longmapsto\mu(x,y^\vee)=\sum_{e\in E}[x_e, y_e^\vee]
\endaligned
\ee
Thus:

\begin{cor}\label{iso-fi}
Proposition {\rm\ref{et-kuran} } gives  a local $G$-equivariant
isomorphism
$(U,y)\cong (\mu_{ \mathbf n}^{-1}(0), 0)$
where $[U/G]$ is a slice \'etale for the point $[F]\in \M_\sigma(v,\D)$, and $y$ is a point of $U$ mapping to $[F]$. Moreover {\rm(}see {\rm(}\ref{utov}{\rm))}, there is a $G$-invariant analytic isomorphism: 
\be\label{localGiso}
\phi: \U\to \V, \quad \text{with} \quad  \phi(F)=0
\ee
where $\U$, and $\V$ are  $G$-invariant saturated open analytic neighbourhood of $y$ in $U$ and, respectively, 
of $0$ in $\mu_{ \mathbf n}^{-1}(0)$.
\end{cor}

\subsection{Walls and Chambers relevant to the $\sigma_0$-polystable point $[F]\in M_\sigma(v, \D)$}
${}$
We keep the notation (\ref{notationF}).
We consider the map:
$$
\Z: \Stab^\dagger(X)\longrightarrow K_{\num}(\D)^\vee\otimes\CC
$$

where, for $\sigma=(\A, Z)$,
$$
\Z(\sigma)(-)=\langle \Omega_Z, - \rangle \,,\qquad \text{i.e.}\quad Z(-)=\langle \Omega_Z, -\rangle
$$
This map  is \'etale over an open subset of $ K_{\num}(\D)^\vee\otimes\CC$.
Let
$$
\V\subset\Stab(X)
$$ 
be a small neighbourhood of 
$\sigma_0$ mapped  homeomorphically onto its image via 
$\Z$. In this neighbourhood a stability condition $\sigma=(\A, Z)$ is completely determined by the stability function $Z$, so that we will often  write "$Z\in \V$", instead of $\sigma\in \V$.
Let $v_1,\dots,v_s$ be as above. Consider the sublattice 
$$
\Lambda=\langle v_1,\dots,v_s\rangle\subset  K_{\num}(\D)
$$
generated by $v_1,\dots,v_s$. 

Define 
$$
\wt\Lambda:= \ZZ v_1\oplus\cdots\oplus\ZZ v_s
$$
and let
$$  
\eta:\wt\Lambda\longrightarrow \Lambda
$$
be the natural projection. 
Define
$$
\aligned
\gamma:&\quad\V
\longrightarrow \Hom(\wt\Lambda, \RR)\equiv\RR^s\\
\endaligned
$$
by
$$
\gamma(Z)\left(\sum_ia_iv_i\right)=\sum_ia_i\Im m\frac{Z(v_i)}{Z_0(v)}
$$
(Further down, we will write this sum as:
$\sum_ia_i\deg_{\sigma,v}(F_i)$).

Since $F_i$ has the same $Z_0$-phase as $F$, we have that $\Im m\left(\frac{{Z_0}(v_i)}{Z_0(v)}\right)=0$, so that
$$
\gamma(Z_0)\equiv(0,\dots,0)\in\RR^s
$$
We next restrict to $\V$ the subdivision into walls and chambers  of $\Stab(S)$, but we only look at the
{\it $v$-walls that are relevant to $F$}, i.e. those walls in $\W(v)$ whose points represent stability conditions making $F$ strictly semistable. We denote these walls  by the symbols
 $\H_{w_\alpha}(v)$, where
$$
w_\alpha=\sum_{i=1}^s\alpha_iv_i  \,\,\quad 0\leq\alpha_i\leq n_i\
$$
and where
$$
\H_{w_\alpha}(v)=\left\{ Z\in \V\,\,|\,\,\exists\,\, F'\subset F\,\,,\text{s.t.}\,\,\phi_Z(F')=\phi_Z(F)\,\,,v(F')=w_\alpha\right\}.
$$
This means that  $Z\in\H_{w_\alpha}(v)$, if and only if  there exists a sub-object $F'\subset F$, with $v(F')=w_\alpha$, making $F$ strictly semi-stable.
The equation for such a  wall is given by
$$
\H_{w_\alpha}(v)=\{ Z\in\Stab(X)\,\,|\,\,\gamma(Z)(w_\alpha)=0\}
$$
Next, by normalizing stability conditions,  we consider the {\it slice}:
$$
\mathcal S=\left\{Z\in \V\,\,|\,\,\gamma(Z)(v)=0\right\}
$$
The subdivison in walls and chambers (relative to $F$) is just the restriction to $\mathcal S$ of the corresponding subdivision 
of $
\V$.
We finally consider the map 
$$
\aligned
\Xi: \,\, &\mathcal S\longrightarrow \W_{\bf n}\otimes\RR={\bf n}^\perp\otimes \RR^s\\
&Z\mapsto \gamma(Z)
\endaligned
$$
In this map $Z_0$ is sent to the origin, and the walls of $\mathcal S$ relative to $F$ are sent
to the walls of the ext-quiver, namely:
$$
\Xi\,\,(\H_{w_\alpha}(v))=\W_\alpha(\bf n)
$$


\section{Normality and irreducibility}\label{norm}

Let $v\in K_{\num}(\D)$. From \cite[Definition 2.20]{BM-mmp14}, we recall that a wall $\W\subset\Stab^\dag(\D)$, relative to $v$, is called {\it totally semistable} if, for any $\sigma\in \W$, there are no $\sigma$-stable object,
i.e $M_\sigma^{st}(v,\D)=\emptyset$. 
\vskip 0.2 cm
An {\it irreducible component $M$ of $M_\sigma(v,\D)$ is said to be totally semistable if $M^{st}=\emptyset$.} Here $M^{st}$ is the open subset parametrizing stable objects

\vskip 0.2 cm

Our aim is to prove the following theorem:

\begin{theorem}\label{normM} Let $v\in K_{\num}(\D)$, with $v^2>0$,  and let $\sigma\in Stab^\dag(\D)$, not lying on a totally semistable wall. Then the moduli space $M_\sigma(v, \D)$ is normal.
\end{theorem}
\proof Two main ingredients go into the proof of this theorem. The first one is the $G$-equivariant local isomorphism recalled in (\ref{localGiso}). The second is the combination of  Crawley-Boevey's \cite[Theorem 1.2]{C-B01}, and \cite[Theorem 1.1] {C-B03}.

\vskip 0.3 cm
Let $\ov{M^{st}_\sigma(v,\D)}$  be the closure of the open set  $M_\sigma^{st}(v,\D)\subset M_\sigma(v,\D)$ parametrizing $\sigma$-stable  objects. Let $x=[F]\in\ov{M^{st}_\sigma(v,\D)}$ and let $G=G_x$ be  the stabilizer of $F$.  Because of the $G$-equivariant isomorphism (\ref{localGiso}) in Corollary \ref{iso-fi}, it is enough to prove that $\mu_{ \mathbf n}^{-1}(0)\dslash G$ is normal. By   \cite[Theorem 1.1]{C-B03}, $\mu_{ \mathbf n}^{-1}(0)\dslash G$ with its reduced  structure
is normal. By Theorem 1.2 in \cite{C-B01}, $\mu_{ \mathbf n}^{-1}(0)$  is a reduced complete intersection as long as we can find a simple representation in $\mu_{ \mathbf n}^{-1}(0)$.
Since $y=[F]$ is in $\ov{M^{st}_\sigma(v,\D)}$, there are  $\sigma$-stable objects $F'$ that are parametrized by points near $y$ in $\U$. By Corollary \ref{stab-simp} below (whose proof is independent from the normality of $\U$ and $\V$), under the isomorphism 
$\phi$ given in  (\ref{localGiso}), the stable objects $F'$ correspond  to simple representations in $\V\subset\mu_{ \mathbf n}^{-1}(0)$.
It follows that $\ov{M^{st}_\sigma(v,\D)}$ is normal in $[F]$.
\vskip 0.2 cm
We are reduced to proving the following:
\begin{claim}\label{totses}If $\sigma$ is not  on a totally semistable wall for $v$, there is no totally semistable component of $M_\sigma(v,\D)$ .
\end{claim}
\begin{rem}\label{red-norm}\rm{As we observed in \cite[Theorem 1.1]{C-B03},    \cite[p. 58, third paragraph]{BM-mmp14},
Crawley-Boevey proves that $(\mu_{ \mathbf n}^{-1}(0))_{red}$ is always normal.
It follows that $(M_\sigma(v,\D))_{red} $ is always normal, and in particular it is the disjoint  union 
of irreducible components.}
\end{rem}
Let us prove the Claim. Let $\W\in \Stab^\dag(\D)$ be a wall with respect to $v$.
Recall from \cite[Proposition 5.1]{BM-mmp14}, that one associates to $\W$ a primitive,
rank two, sublattice $\H\subset\K_{num}(\D)$, of signature equal to $(-1,1)$, defined by
$$
w\in \H\,\,\iff\ \,\, \forall \,\,\sigma=(Z,\P)\in\W, \quad Z(w) \,\,\text{is collinear with } \,Z(v)
$$
so that the $w$'s in  $\H$ are, potentially,  Mukai's vectors for  subobjects that are making 
objects with Mukai vector $v$ strictly $\sigma$-semistable for all $\sigma\in W$. When needed,
we will write $\H_v$, instead of $\H$, in order to highlight the dependence on $v$.
\cite[Theorem 5.7 ]{BM-mmp14} says that a wall $\W$, 
is totally semistable if and only if, either:
\vskip 0.2 cm
a) there exists a class $w\in \H$ which is isotropic  (i.e. $w^2=0$),  with $\langle v,w\rangle=1$, or:
\vskip 0.2 cm

b) there exists a  class $s\in \H$ which is  effective, and spherical (i.e. $s^2=-2$), with $\langle v,s\rangle<0$.

\vskip 0.2 cm

Let then $M'$ be a totally semistable  irreducible component of $M_\sigma(v,\D)$. This component, like any other, is stratified by locally closed strata whose open part is a locus of objects with fixed polystable type. By the irreducibility of $M'$ there is a maximal stratum, which  must be open. The maximal stratum is the image of an open subset of $\Pi_{i=1}^s\,\, S^{n_i}M_\sigma(v_i,\D)$ under the finite morphism
$$
\aligned
&\Pi_{i=1}^s\,\, S^{n_i}M_\sigma(v_i,\D)\longrightarrow M'\subset M_\sigma(v,\D)\\
\endaligned
$$
for some decomposition $v=\sum_iv_i$.
In this maximal  stratum,
take a point $[F]$  given by
$$
F=\overset{s}{\underset{i=1}\oplus}F_i^{n_i}\,,\,\,\,\,\,F_i\ncong F_j\,,\,\,\,\,\text{for}\,\,i\neq j\,,\,\,\,
\,\,\,\,v(F_i)=v_i\,,\,\,\,\,\,v=\sum_in_iv_i\,,\,\,\,\,\,\,{\bf{n}}=(n_1,\dots,n_s)
$$
The maximality of the stratum will play its role through the following obvious observation:

\be\label{accorpare}
\aligned
&\text{\parbox{.88\textwidth}{{ \it Let ${\bf{m}}=(m_1,\dots,m_s)$, and write ${\bf{m}} \leq{\bf{n}}$, if $m_i\leq n_i$, for all $i$. 
Then, for no ${\bf{m}}\leq{\bf{n}}$ the object
$F'=\overset{s}{\underset{i=1}\oplus}F_i^{m_i}$ has a stable deformation: otherwise $F$ would not be
on the maximal stratum. 
 }}}
 \endaligned
\ee
Following \cite[Definition 5.5 ]{BM-mmp14} we let $P_\H$ denote the positive cone of $\H$, which is generated by the integral classes $u\in \H$, with $u^2\geq0$, and $\langle u,v\rangle>0$.
We next observe that, in the decomposition $v=\sum_in_iv_i$,  there can not be a pair $v_i, v_j$, with $i\neq j$ such that:

1) $v_i^2\geq 2$, $v_j^2\geq 2$, 

nor a pair $v_i, v_j$  such that:

2) $\langle v_i, v_j \rangle\geq 2$, $i\neq j$.

In both cases we would get:
$$
(v_i+v_j)^2+2\,>\,(v_i^2+2)+(v_j^2+2)\,.
$$
To prove this inequality, regarding case 1), one  uses  \cite[Lemma 6.4]{BM-mmp14}, and case 2) is similar. 
Consider the ext-quiver associated to $F_i\oplus F_j$, and the corresponding moment map $\mu_{(1,1)}$
By  \cite[Theorem 1.2] {C-B01} , and recalling the dimension formula (\ref{dim-quiv}), the last inequality  says  that $\mu^{-1}_{(1,1)}(0)$
contains a simple representation. By our local correspondence between quiver varieties and moduli spaces we get that $F_i\oplus F_j\in\ov{M^{st}_\sigma(v_i+v_j, \D)}\neq\emptyset$, against (\ref{accorpare}).
In conclusion, in the expression $v=\sum_in_iv_i$ we have at most one $v_i$ with $v_i^2>0$,
and we also have $\langle v_i, v_j \rangle\leq 1$, for $i\neq j$. On the other hand, since the $F_i$'s are $\sigma$-stable, we must have $\langle v_i, v_j \rangle\geq0$, so that:
$$
0\leq\langle v_i, v_j \rangle\leq 1\,,\quad i\neq j\,.
$$  
Notice that, exactly in the same vein, i.e. using (\ref{accorpare}), for $F_i^{n_i}$ or for $F_i^{n_i}\oplus F_j^{n_j}$
one sees that 
\be\label{consec}
v_i^2>0\,\,\Rightarrow\,\, n_i=1\,,
\ee

Also observe that, for each $i$, the $\sigma$-stability of $F_i$ 
implies that  $\sigma$ does not lie on a totally semistable $v_i$-wall. 
Now suppose that 
 $v_{j}^2=0$, for some $j\neq i$. Then, for all $i$, we have  $\langle v_i,v_j \rangle=0$, otherwise there would be an index $i$ for which $\langle v_i,v_j \rangle=1$, and   \cite[Theorem 5.7]{BM-mmp14} would imply that $\sigma$ lies on a totally semistable $v_i$-wall. It follows that such a $v_j$ is an isolated vertex
 for the quiver graph $Q_F$. But then, always using  our local identification (\ref{localGiso}), we see that
 a neighbourhood  of $F$ in $M'$ would decompose into a product where one of the factors is isomorphic to an open neighbourhood of  $ [F_j^{n_j}]\in S^{n_j}M_\sigma(v_j,\D)$, which is normal (as $M_\sigma(v_j,\D)$ is a normal surface  since $v_j^2=0$). Thus  we may as well assume that: 
$$
v=aw+\sum_ic_is_i\,,\quad\,\,\,\,\,a\in\{0,1\}\,,\,\,w^2>0\,,\,\,c_i\geq0\,,\,\,s_i^2=-2
$$
We may also assume that $a=1$, otherwise $M'$ would reduce to a  point.
We may also assume that at least one of the $c_i$'s is not equal to zero, otherwise
$M'=\ov{M_\sigma(w,\D)}$, which is normal.
Let $Q_1,\dots, Q_k$ be the connected components of the graph, obtained from the ext-graph $Q_F$ of $F$,
by taking out the vertex $w$. We claim that the genus of  each $Q_h$ is zero. With abuse of notation,
let $s_1,\dots,s_t$ be the vertices of $Q_h$.
We have:
$$
\left(\sum_{i=1}^t s_i\right)^2=\sum_{i=1}^t s_i^2+2\sum_{i<j}\langle s_i, s_j\rangle=2g(Q_h)-2
$$
so that, if $g(Q_h)\geq1$ we get:
$$
\left(\sum_{i=1}^t s_i\right)^2+2>\sum_{i=1}^t( s_i^2+2)
$$
By Crawley-Boevey's   \cite[Theorem1.2]{C-B01}, and our local identification (\ref{localGiso}),
this contradicts (\ref{accorpare}). Next observe that we may assume that the ext-graph is connected. In fact, if we had $Q=\sqcup \Q_i$,
then $\mu^{-1}_Q(0)\cong\prod \mu^{-1}_{\Q_i}(0)$ and the local
isomorphism yields a corresponding product decomposition of a neighbourhood of $[F]\in M_\sigma(v,\D)$ and the reasoning would be component-wise.
Assuming then that the ext-graph is connected, each $Q_h$
is connected to $w$, and considering the subgraph $Q_h\cup\{w\}$, we can see, with the same method used above, that also this graph has genus zero. Using (\ref{consec}) we get:
\be\label{tss}
v=w+\sum_is_i
\ee

Now take  a vertex $s_l$ in this graph which is connected just to another  vertex $u$.
The vertex $u$ can be equal to $w$ or to one of the $s_j$'s. 
Then
$$
\langle v, s_i \rangle=s_i^2+\langle u, s_i \rangle=-1
$$
By point b) above we conclude that $\sigma$ sits in a totally semistable wall.
\endproof
\begin{rem} \label{tot-semist}From the proof above and from formula (\ref{tss}),we see that if $M'$ is a totally semistable component of $M_\sigma(v,\D)$, then
$$
(M')_{red}\cong M_\sigma(w,\D)\times\prod_i M_\sigma(s_i,\D)
$$
\end{rem}

\vskip 0.3 cm

We are now in the position to prove the following theorem.

\begin{theorem} \label{irrM}Let $v\in K_{\num}(\D)$, with $v^2>0$,  and let $\sigma\in Stab^\dag(\D)$, not lying on a totally semistable wall. Then the moduli space $M_\sigma(v, \D)$ is irreducible.
\end{theorem}
\proof 
When the Mukai vector $v$ is primitive and $\sigma$ is $v$-generic, 
the irreducibility of moduli spaces of Bridgeland stable objects on a K3 surface has been established by \cite{Tod08}, \cite{BM-pr14}, \cite{Yos01}. Under the same hypotheses for $v$ and $\sigma$, the irreducibility is proved in \cite{BLM${}^+$19}, when $\D=\Ku(X)$, where $X$ is a cubic fourfold, and in \cite{LPZ20}, when $v$ is of OG10-type. The remaining cases ($v$ non-primitive, $\sigma$ $v$-generic and $\D$ coming from a K3 or from a Kuznetsov component) are proved in \cite{Sa23}.
Here $v$ is arbitrary and  we are no longer assuming that $\sigma$
be  $v$-generic, but only that it does not lie on a totally semistable wall.
By the previous theorem we know that $M_\sigma(v, \D)$ is normal and it is therefore the disjoint union of its irreducible components. Choose a $v$-generic stability
condition  $\sigma'$ close to $\sigma$. Then $M_{\sigma'}(v, \D)$ is irreducible and surjects onto $\ov M^{st}_\sigma(v, \D)$, proving the Theorem.
\endproof


\vskip 0.3 cm
\section{Variation of GIT}\label{VGIT}

We fix a Mukai vector $v$, and  a Bridgeland stability
condition $\sigma_0=(\A_0, Z_0)\in \Stab^\dagger(\D)$, not lying on a totally semistable wall, so that $\M_{\sigma_0}(v,\D)$ is normal , as we proved in the preceding section.
As in  (\ref{ell}), we  consider the cohomology class 
$\ell_0:=\ell_{\sigma_0}\in H^2(\M_{\sigma_0}(v,\D),\RR)$, and we recall the fundamental equality 
(\ref{git-slope}). For brevity we  write $\M_{\sigma_0}(v,\D)$ instead of $\M^{ss_{\ell_0}}_v(\A)$ and
we will use, intercheangeably,  the expressions {\it $\sigma_0$-semistable}, and {\it $\ell_0$-semistable}.
Next we choose a  stability condition $\sigma=(\A_0, Z)$,  very close to $\sigma_0$, (so that 
$\M_{\sigma}(v,\D)$ is  normal),
and we set $\ell=\ell_\sigma$. Consider a slice \'etale
$$
\epsilon': [U/G]\longrightarrow \M_{\sigma_0}(v,\D)
$$
at a point $x=[F]\in \M_{\sigma_0}(v,\D)$, where $G=G_x$.
For $u\in U$, let $\ov u$ denote  the image of $u$ in $[U/G]$.
Set
\be\label{usigma}
U^\sigma=\{ u\in U\,\,|\,\,\epsilon'(\ov u)\,\,\text{is}\,\,\sigma\text{-semistable} \}
\ee
We  have 
a commutative diagram, where the outer square is given by the slice \'etale 
$$
\xymatrix{[U/G]\ar[ddd]_{\rho_0}\ar[rrr]^{\epsilon'}&&&\M_{\sigma_0}(v,\D)\ar[ddd]^{h_0}\\
&[U^\sigma/G]\ar[ddl]_{\rho'}\ar[r]^{\epsilon''}\ar[ul]_\tau\ar[d]_p&\M_\sigma(v,\D)\ar[ddr]^{h'}\ar[d]^\pi\ar[ur]^\eta\\
&U^\sigma\dslash G\ar[r]^\alpha\ar[dl]_\rho&M_\sigma(v,\D)\ar[dr]^h\\
U\dslash G\ar[rrr]^\epsilon&&&M_{\sigma_0}(v,\D)
}
$$
Also set:
$$
\rho'=\rho\circ p\,,\quad h'=h\circ\pi
$$
\begin{claim}\label{cart-trap} In the diagram above the bottom trapeze is cartesian.
\end{claim}
\proof As $M_\sigma(v,\D)$,  and $U\dslash G$ are both normal, it suffices to prove that $\alpha$ is bijective or, equivalently, that $\alpha$ establishes  a bijection between $\rho^{-1}(z)$ and  $h^{-1}((\epsilon (z))$, for every $z\in U\dslash G$ . The outer square and the top trapeze are both cartesian and $\eta$ is an open immersion.  The inner square is cartesian too, by virtue of  \cite{AHLH22}, as soon as we realize that $\epsilon''$ is $\Theta$-surjective, which it is since $\epsilon'$ is $\Theta$-surjective, and the top trapeze is cartesian. Now $\epsilon''$ establishes  a bijection between
$\tau^{-1}\rho_0^{-1}(z)=\rho^{-1}p^{-1}(z)$ and $h_0^{-1}(\epsilon(z))\cap\M_\sigma(v)=\pi^{-1}h^{-1}(\epsilon(z))$, but then $\alpha$ establishes the desired bijection.
\endproof
\begin{rem}It is worth observing that the proof of this Claim is much shorter than the proof of the analogous   \cite[Proposition 7.2]{AS18}.
\end{rem}

\vskip 0.3 cm

Look at the locus $U^\sigma$ defined in (\ref{usigma}). Pulling back $\ell_\sigma$ form $M_\sigma(v)$ to  $U$
 gives a $G$-linearized line bundle $\mathcal L_\sigma$ on $U$. The fundamental comparison result given in 
 (\ref{bridg-git}) tells us  that
  $$
  U^\sigma=U^{s_{\mathcal L_\sigma}}=\{z\in U\,\,|\,\,z\,\,\text{is $\Theta$-semistable with respect to the class $\ell_\sigma$ }\,\}
  $$
  \vskip 0.3 cm
 Let us now recall a few facts from   \cite[Section 5.1]{KKV13}.  Let $G$ be a connected, reductive algebraic group.
 Denote by $\X(G)$ the group of characters of $G$.  If $V$ is an irreducible variety acted on by $G$, denote by $\Pic_G(V)$  the group of $G$-linearized line bundles on $V$.  By \cite[Proposition 1.3, Lemma 2.2 and Proposition 2.3]{KKV13} 
there is a surjection $\X(G)\twoheadrightarrow H^1_{alg}(G,\cO(V)^*)$ induced by the injection $\CC^*\hookrightarrow\cO(V)^*$, and a natural injection
 $H^1_{alg}(G,\cO(V)^*)\hookrightarrow\Pic_G(V)$. In particular to every character $\chi\in \X(G)$ one associates a
 $G$-linearized line bundle on $\mathcal L_\chi$ on $V$. Moreover  there is an exact sequence 
\be\label{picG}
 1\longrightarrow \Pic(V\sslash G)\overset{\pi^*}\longrightarrow\Pic_G(V)\overset\delta\longrightarrow \prod_{x\in\C}\X(G_x)
\ee
 where the index set $\C\subset V$ is a set of representatives of the closed orbits in $V$, and the map $\delta$ is defined by associating to a line bundle $\mathcal L $ the characters of the isotropy groups on the fibers $\mathcal L_x$, for $x\in\C$.
  \vskip 0.3 cm
  Let us go back to (\ref{effe}), (\ref{notationF}) and to the slice \'etale $U$ at $[F]\in M_\sigma(v)$.
  Consider the character $\chi_\sigma\in \X(G)$ defined by
  $$
  \chi_\sigma((g_1,\dots,g_s))=\prod_{j=1}^s\det(g_j)^{-\R e(Z(v_j))}
  $$
  \begin{prop} $\delta(\mathcal L_\sigma)= \delta(\mathcal L_{\chi_\sigma})$
  \end{prop}

  \proof   The proof follows exactly the same steps in the proof of \cite[Lemma 7.17]{AS18}.
  \endproof
  
From the exact sequence (\ref{picG}), it follows that $(\mathcal L_{\chi_\sigma})(\mathcal L_{\sigma})^{-1}\in 
 \Pic(V\sslash G)$, and as a consequence we have the following result.
 
  \begin{cor}
  $U^{s_{\mathcal L_\sigma}}=U^{s_{{\chi_\sigma}}}(=U^\sigma)$ and $\,\,U^{s_{\mathcal L_\sigma}}\sslash G=U\sslash_{\chi_{\sigma}}G$.
  \end{cor}

  \begin{cor} \label{stab-simp}
Under the isomorphism $\phi:(\U, [F])\to (\V,0) $ given in  (\ref{localGiso}), points in $\U$
  corresponding to  polystable objects in $\M_\sigma(v,\D)$ go to points in 
  $\V$ corresponding to direct sum of simple representations.
  In particular there exists a simple representation in a neighborhood
  of $0\in \V$, if and only if there exists a stable object in a neighborhod of $[F]$.
  \end{cor}
  
 By Claim \ref{cart-trap} we  have a cartesian diagram
 \be\label{loc-VGIT}
\xymatrix{
 M_{\sigma}(v,\D) \ar[d]_\rho&U\sslash_{\chi_\sigma}G\ar[d]^h\ar[l]_\eta\,,   \\
M_{\sigma_0}(v,\D) &U\dslash G\ar[l]^\epsilon    
} \qquad U\sslash_{\chi_\sigma}G=U^\sigma\sslash G
\ee
Of course the left vertical arrow  must be compared with the variation of GIT quotients
$$
\xymatrix{
\mu^{-1}(0)\sslash_{\chi_\sigma}G\ar[d]^h\\
\mu^{-1}(0)\dslash G   
} \qquad\qquad\qquad\qquad\qquad
$$
 This is best done in the analytic category. The local analytic $G$-equivariant  isomorphism $\phi:\U\to\V$, given in (\ref{localGiso}), together with diagram (\ref{loc-VGIT}), 
give  a commutative diagram 
 
$$
\xymatrix{
U^\sigma \sslash G \ar[d] & \ar[d] \ar@{_{(}->}[l] \U \sslash_{\chi_\sigma} G \ar[r]^\sim\ar[d]^{h} &  \ar[d]\V \sslash_{\chi_\sigma} G  \ar@{^{(}->}[r] \ar[d]^{\bf h}& \ar[d]  \mu^{-1}(0)  \sslash_{\chi_\sigma} G\\
U\sslash G& \ar@{_{(}->}[l]\U \sslash G \ar[r]^\sim & \V\sslash G  \ar@{^{(}->}[r]  & \mu^{-1}(0)  \sslash G
} 
$$
showing that the variation of GIT quotients $\bf h$, given by wall crossing for quivers
corresponds exactly to the variation $h$ of moduli spaces around a wall in
$\Stab^\dag(\D)$.

\endproof


 \bibliographystyle{amsalpha}
 \bibliography{ASQuiv}

\end{document}